\documentclass[reqno, a4paper]{amsart} 
\usepackage{latexsym}
\usepackage[english]{babel}
\usepackage{fancyhdr}
\usepackage[mathscr]{eucal}
\usepackage{amsmath}
\usepackage{mathrsfs}
\usepackage{amsthm}
\usepackage{amsfonts}
\usepackage{amssymb}
\usepackage{amscd}
\usepackage{bbm}
\usepackage{graphicx}
\usepackage{graphics}
\usepackage{latexsym}
\usepackage{color}
\theoremstyle{plain}
\newtheorem{theorem}{Theorem}[section]
\newtheorem{lemma}[theorem]{Lemma}
\newtheorem{corollary}[theorem]{Corollary}
\newtheorem{proposition}[theorem]{Proposition}

\theoremstyle{definition}

\newcommand {\absleq} {{\leq_{|\, \cdot\, |}\, }}
\def\tD {{\tilde{\mathcal D}}}

\def\lbeq(#1){\label{eqn:#1}}
\def\refeq(#1){{\rm (\ref{eqn:#1})}}
\def\lbth(#1){\label{th:#1}}
\def\refth(#1){{\rm Theorem \ref{th:#1}}}
\def\refthb(#1){{\bf Theorem \ref{th:#1}}}
\def\lblm(#1){\label{lm:#1}}
\def\reflm(#1){{\rm Lemma \ref{lm:#1}}}
\def\reflms(#1,#2){{\rm Lemmas \ref{lm:#1} and \ref{lm:#2}}}
\def\reflmss(#1,#2,#3){{\rm Lemmas \ref{lm:#1}, \ref{lm:#2} and \ref{lm:#3}}}
\def\reflmsss(#1,#2,#3,#4){{\rm Lemmas \ref{lm:#1},\, \ref{lm:#2},\, \ref{lm:#3} and \ref{lm:#4}}}
\def\reflmb(#1){{\bf Lemma \ref{lm:#1}}}
\def\lbprp(#1){\label{prp:#1}}
\def\refprp(#1){{\rm Proposition \ref{prp:#1}}}
\def\refprpb(#1){{\bf Proposition \ref{prp:#1}}}
\def\lbcor(#1){\label{cor:#1}}
\def\refcor(#1){{\rm Corollary \ref{cor:#1}}}
\def\refcors(#1,#2){{\rm Corollaries \ref{cor:#1} and \ref{cor:#2}}}
\def\lbrm(#1){\label{rm:#1}}
\def\refrm(#1){{\rm Remark \ref{rm:#1}}}
\def\lbass(#1){\label{ass:#1}}
\def\refass(#1){{\rm Assumption \ref{ass:#1}}}
\def\lbdf(#1){\label{df:#1}}
\def\refdf(#1){{\rm Definition \ref{df:#1}}}
\def\lbsec(#1){\label{s:#1}}
\def\refsec(#1){{\rm \S\ref{s:#1}}}
\def\lbsubsec(#1){\label{ss:#1}}
\def\refsubsec(#1){{\rm \S\ref{ss:#1}}}

\def\tGa{\widetilde{\Gamma}}
\def\Gg{{\mathcal G}}

\newcommand{\ii}{\mathrm{i}}

 \newcommand{\ud}{\mathrm{d}}

\newcommand{\lam}{\lambda}
\def\ab{{\bf a}}
\def\bb{{\bf b}}
\def\Bb{{\bf B}}

\def\eb{{\bf e}}
\def\fb{{\bf f}}

\def\ph{{\varphi}}
\def\bqn{\begin{equation}}
\def\eqn{\end{equation}}
\def\C{{\mathbb C}}
\def\N{{\mathbb N}}
\def\Z{{\mathbb Z}}
 \def\Cb{{\overline{\mathbb C}}}

\def\R{{\mathbb R}}

\def\a{\alpha}

\def\c{\gamma}
\def\Ga{\Gamma}
\def\d{\delta}
\def\Dg{{\mathcal D}}
\def\Eg{{\mathcal E}}
\def\Sg{{\mathcal S}}

\def\p{\psi}

\def\ep{\varepsilon}

\def\th{\theta}

\def\m{\mu}

\def\r{\rho}
\def\s{\sigma}

\def\w{\omega}
\def\W{\Omega}
\def\Hg {{\mathcal H}}
\def\la{\langle}
\def\ra{\rangle}

\def\lap{\Delta}
\def\ax{{\la x \ra}}
\def\pa{{\partial}} 

\begin{document}

\title[Two dimensional point interactions]
{Two dimensional Schr\"odinger operators with point interactions: 
threshold expansions, zero modes and $L^p$-boundedness of wave operators}

\author[H.D.~Cornean]{Horia D. Cornean}
\address[H.D.~Cornean]{Department of Mathematical Sciences, Aalborg University \\ Skjernvej 4A, 9220 Aalborg (Denmark).}
\email{cornean@math.aau.dk}

\author[A.~Michelangeli]{Alessandro Michelangeli}
\address[A.~Michelangeli]{International School for Advanced Studies -- 
SISSA \\ via Bonomea 265 \\ 34136 Trieste (Italy).}
\email{alemiche@sissa.it}

\author[K.~Yajima]{Kenji Yajima}
\address[K.~Yajima]{Department of Mathematics \\ Gakushuin University 
\\ 1-5-1 Mejiro \\ Toshima-ku \\ Tokyo 171-8588 (Japan).}
\email{kenji.yajima@gakushuin.ac.jp}

\begin{abstract}
We study the { threshold behaviour} of two dimensional Schr\"odinger operators with 
{ finitely many} local point interactions. We show that the { resolvent can  
either be continuously extended} up to the threshold, 
in which case we say { that} the operator is of regular type, 
or it has singularities associated with { 
$s$ or $p$-wave resonances or even with an embedded eigenvalue at zero, for whose existence we give necessary and sufficient conditions. An embedded eigenvalue at zero may appear only if we have at least three centres.}
 
{ When} the operator is of regular type we 
prove that the wave operators are bounded in $L^p(\R^2)$ 
for all $1<p<\infty$. 
{With a single center we always are in the  regular type case. }
\end{abstract}

\date{\today}

\maketitle

\section{Introduction and main results}\label{sec:theorems} 

Let $Y=\{y_1,\dots,y_N\}$ be $N$ points in the plane $\R^2$ { with  $1\leq N<\infty$. Let} $T_0$ be the densely defined non-negative symmetric operator in the 
Hilbert space $L^2(\R^2)$ defined by 
\[
T_0 \colon = -\lap\Big\vert_{C_0^\infty(\R^2 \setminus Y )} , 
\quad \lap = \pa^2/\pa x_1^2 + \pa^2/\pa x_2^2,\quad {x=(x_1,x_2)\in \R^2}.
\]
A Schr\"odinger operator on $\R^2$ with point interactions at $Y$ is 
any selfadjoint extension of $T_0$. In this paper we are concerned 
with  local point interactions at $Y$ 
which are parametrized by { the interaction strengths} 
$\alpha=(\alpha_1,\dots,\alpha_N)\in \R^N $. { The corresponding operators are denoted by $H_{\a,Y}$ 
and} are defined via the resolvent equation: 
\begin{equation}\lbeq(eq:resolvent_identity)
(H_{\a,Y} -z^2)^{-1} -(H_0-z^2)^{-1} \;=\; \sum_{j,k=1}^N 
\{\Gamma_{\alpha,Y}(z)\}^{-1}_{jk} \,\mathcal{G}_{z}(\cdot-y_j)
\otimes \overline{\mathcal{G}_{z}(\cdot-y_k)},
\end{equation} 
where $z\in \C^{+}= \{z\in \C \colon \Im z>0\}${; more details on the right hand side of \eqref{eqn:eq:resolvent_identity} are given next}.   
$\mathcal{G}_z(x)$ 
is the convolution kernel of $(-\lap -z^2)^{-1}$ in $L^2(\R^2)$:
\begin{equation}\lbeq(1-1)
\mathcal{G}_z(x)
= \frac1{(2\pi)^2}\int_{\R^2}\frac{e^{ix\xi}d\xi}{\xi^2-z^2} 
\end{equation}
and, in terms of the Hankel function of the first kind (see e.g. (10.8.2) of \cite{DLMF}), 
\bqn 
\mathcal{G}_z(x)= \frac{i}4 H_0^{(1)}(z|x|)
\eqn 
where 
\begin{multline} \lbeq(hankel)
\frac{i}4 H_0^{(1)}(z)= \Big(-\frac1{2\pi}\log\Big(\frac{z}{2}\Big)+ \frac{i}4 -
\frac{\gamma}{2\pi}\Big) J_0(z) \\
-\frac1{2\pi}\left(\frac{\frac14 z^2}{(1!)^2}
-\Big(1+\frac12\Big)\frac{(\frac14 z^2)^2}{(2!)^2} + \Big(1+\frac12+\frac13\Big)
\frac{(\frac14 z^2)^3}{(3!)^2} - \cdots \right),
\end{multline}
and $J_0(z)$ is the $0$-th order Bessel function 
\bqn  \lbeq(bessel)
J_0(z) = 
\sum_{k=0}^\infty \frac{(-1)^k}{(k!)^2}\Big(\frac{z^2}{4}\Big)^k. 
\eqn 
$\Gamma_{\a,Y}(z)$ is the $N \times N$ matrix whose $(j,k)$-entry is the 
function of $z\in \Cb^{+}\setminus \{0\}$ given by 
\begin{equation}\lbeq(ga-def)
\Gamma_{\alpha,Y}(z)_{jk}= 
\Big(\alpha_j+ \frac1{2\pi}\log \Big(\frac{z}{2}\Big) - \frac{i}{4} 
+ \frac{{ \gamma}}{2\pi}\Big)\delta_{jk} - 
\mathcal{G}_z(y_j-y_k)\hat{\delta}_{jk},
\end{equation}
where $\d_{jk}$ is the Kronecker delta, 
$\hat{\d}_{jk}=1-\d_{jk}$ and { $\gamma$} the Euler constant. 

The following { facts are} well known (see \cite{albeverio-solvable}, pp. 163-165).
\begin{enumerate} 
\item[{(1)}]  The equation 
\refeq(eq:resolvent_identity) defines a unique selfadjoint operator 
$H_{\a,Y}$ in $L^2(\R^2)$ with domain  
\bqn \lbeq(dom)
\Big\{
u(x)=v(x)+ \sum_{j,k=1}^N [\Gamma_{\alpha,Y}(z)^{-1}]_{jk} v(y_k)
\mathcal{G}_{z}(x-y_j)
\colon  v \in H^2(\R^2)\Big\}
\eqn   
which is independent of $z \in \C^{+}$ { whenever $\Gamma_{\alpha,Y}(z)^{-1}$ exists}. 
\item[{(2)}] Given $z$, the function $v \in H^2(\R^2)$ of 
\refeq(dom) is uniquely determined by $u\in D(H_{\a,Y})$ and  
\[
(H_{\a,Y}-z^2)u = (-\lap - z^2) v.
\]
\item[{(3)}] $H_{\a,Y}$ is a real local operator, viz. $H_{\a,Y}u$ is real if $u$ is real, 
and  {if $u=0$ in an open set $U$, then $H_{\a,Y}u=0$ in $U$}. 
\item[{(4)}] The spectrum of $H_{\a,Y}$ consists of an absolutely 
continuous  part $[0,\infty)$ denoted  in short with AC, 
and at most $N$ non-positive eigenvalues.  
Positive eigenvalues and singular continuous spectrum are absent. 
\item[{(5)}] $H_{\a,Y}$ is a rank $N$ perturbation of $-\lap$ and, by virtue of  
Kato-Birman-Rosenblum theorem (\cite{Kato}), the wave operators $W_{\pm}$ 
defined as the strong limits in $L^2(\R^2)$,
\bqn \lbeq(limit)
W_{\pm}= \lim_{{ t\to\pm \infty}} e^{itH_{\a,Y} }e^{-itH_0} 
\eqn 
exist and are complete in the sense that ${\rm Range}\ W_\pm =L^2_{ac}(H_{\a,Y})$, 
the AC subspace of $L^2(\R^2)$ for $H_{\a,Y}$.  Hence 
\[
W^\ast_\pm W_\pm = 1, \quad W_\pm W_\pm^\ast = P_{ac}(H_{\a,Y}), 
\]
where $P_{ac}(H_{\a,Y})$ is the orthogonal projection onto $L^2_{ac}(H_{\a,Y})$. 
The wave operators satisfy the intertwining property  
\bqn \lbeq(inter) 
f(H_{\a,Y})P_{ac}(H_{\a,Y})= W_\pm f(H_0) W_\pm 
\eqn 
for any Borel function $f$ on $\R$. 
\end{enumerate}

The Hankel function has the following integral representation 
\begin{equation} \lbeq(Hank)
\mathcal{G}_z(x) = (i/4)H_0^{(1)}(z|x|)  
= \frac{e^{iz|x|}}{2^{\frac32}\pi} 
\int_0^\infty e^{-t}t^{-\frac12}\left(\frac{t}2-iz|x|\right)^{-\frac12}dt, 
\end{equation}
see formula (3) { on} page 168 of Watson(\cite{Watson}). From \refeq(Hank), we see that 
for any $c>0$, 
\bqn \lbeq(large-hankel) 
\mathcal{G}_z(x)= e^{iz|x|}\w(z|x|), \quad 
|\pa^\a_{\lam} \w(\lam)| \leq C_\a  \la \lam \ra^{-\frac12-|\a|}, 
\quad \lam \in \R, \  |\lam |\geq c,
\eqn 
viz. $(1-\chi(\lam))\w \in S^{-\frac12}(\R)$ (i.e. the space of one dimensional symbols of order $-1/2$)  
whenever $\chi\in C_0^\infty(\R)$ is such that $\chi(\lam)=1$ near { $\lam=0$. Also, for purely imaginary $z=ia \in \C^{+}$ with $a>0$ we have that  
$\mathcal{G}_{ia}(x)$ is positive}, hence 
$\Ga_{\alpha,Y}(ia)$ is real and symmetric.  

If $\s\in \R$, let $L^2_{\s}=L^2_{\s}(\R^2, \ax^{2\s}dx)$ be 
the weighted $L^2$ space 
and $\Bb_\s= \Bb(L^2_\s, L^2_{-\s})$ the Banach space of bounded operators from $L^2_\s$ 
to $L^2_{-\s}$.  

{ Let $\Eg\subset i[0,\infty)$ denote the finite set of 
square roots of negative eigenvalues of $H_{\a,Y}$. Let $\s>1/2$ and $z\in \C^{+}\setminus \Eg$}. The celebrated Agmon-Kuroda theory 
\cite{Ag,Ku} of limiting absorption principle for $(-\lap -z^2)^{-1}$ and the properties 
of the Hankel function \eqref{eqn:hankel}, \eqref{eqn:large-hankel}, { imply that the} $\Bb_\s$-valued analytic function 
$(H_{\a, Y}-z^2)^{-1}$ admits a  
boundary value $(H_{\a, Y}-\lam^2)^{-1}$  { for $\lambda\in \R \setminus \{0\}$  which is} locally H\"older continuous. However, 
it can be singular at $\lam=0$. We shall show 
that $(H_{\a, Y}-\lam^2)^{-1}$ can either be continuously extended to the whole closed 
half plane $\Cb^{+}$, in which case we say $H_{\a,Y}$ is of regular type, or it has singularities 
of one of the three kinds associated with resonances 
of $s$-wave or $p$-wave types or zero energy eigenvalue. 
In the regular case we then show that the wave operators are bounded in $L^p(\R^2)$ for all $1<p<\infty$.  
We write $\lam$ instead of $z$ when we want to emphasize that $\lam$ is in 
$\Cb^{+}\setminus \{0\}$ not only in $\C^{+}$. 

For stating { our main results we need some more notation}. We introduce the vectors  
$$
\widehat{\Gg}_{\lam,Y}(x) 
= \begin{pmatrix} \Gg_\lam (x-y_1) \\ 
\vdots \\ 
\Gg_\lam (x-y_N) \end{pmatrix} , \quad 
\widehat{G}_{0,Y}(x) = \begin{pmatrix} G_{0}(x-y_1)  \\ 
\vdots \\ 
G_0(x-y_N) \end{pmatrix}, 
$$
where $G_{0}(x)$ is the Green function of the 2-dimensional $-\lap$:
\[
G_0(x) = -\frac1{2\pi} \log |x|, \quad (-\lap)^{-1} u(x)= \int_{\R^2} G_0(x-y) u(y) dy, 
\]
so that the right hand side of \refeq(eq:resolvent_identity) may be expressed as 
\bqn  \lbeq(prob-1)
D(\lam, x,y)  = 
\la \widehat{\Gg}_{\lam,Y}(x), \Gamma_{\alpha,Y}(\lam)^{-1} \widehat{\Gg}_{\lam,Y}(y) \ra, 
\eqn 
where $\la \ab , \bb \ra= \sum a_j b_j$ (without complex conjugation). Also:
\begin{align}\label{horia1}
\eb =  
\frac{1}{\sqrt{N}}{\bf \hat{1}}, \quad  
{\bf \hat{1}}= \begin{pmatrix} 1\\  \vdots \\ 1 
\end{pmatrix}, \quad 
P = \eb \otimes \eb, \quad S=1-P.
\end{align}
{ Moreover,} $\tD=\Dg(\a,Y)$ and $\Gg_1(Y)$ are $N \times N$ real symmetric matrices given by 
\bqn \lbeq(deftD)
\tD= \Big(\d_{jk}\a_j + \frac{\hat{\d}_{jk}}{2\pi}\log|y_j-y_k|\Big), \quad 
\Gg_1(Y)= - \Big(\frac{\hat\d_{jk}}{4N}|y_j-y_k|^2\Big).
\eqn 
For an integral operator $K$, 
 { we denote by $K(x,y)$} its integral kernel and we often identify $K$ { with} $K(x,y)$. 
We { will} use the function 
\bqn \lbeq(g)
g(\lam)
= - \frac1{2\pi}\log\Big(\frac{\lam}{2}\Big)+ \frac{i}4 - \frac{\gamma}{2\pi}
\eqn 
which appears in front of $J_0(z)$ in \refeq(hankel) as one of 
the scales for the asymptotic expansions as $\lam \to 0$, the 
other being $\lam$. We have for small 
$|\lam||x|$ that 
\bqn \lbeq(G0g)
\Gg_\lam(x) = g(\lam) + G_0(x) + O(\lam^2|x|^2 g(\lam |x|)).
\eqn 

The representation of any point 
$x\in\mathbb{R}^2$ in polar coordinates will be $x=r\omega$, 
where $r\equiv|x|\geqslant 0$ and $\omega\in\mathbb{S}^1$.
For $u, v\in L^2(\mathbb{R}^2)$, 
$u\otimes v$ denotes the rank-1 operator 
$f \mapsto u\langle v, f\rangle$, where $\langle\cdot,\cdot\rangle$ 
is the usual scalar product in $L^2(\mathbb{R}^2)$, 
anti-linear in the first entry and linear in the second. 
The notation $\langle f, g\rangle$ will { also} be used whenever 
the dual product is meaningful, say for $f \in \Sg$ and $g\in \Sg'$.
For the Fourier transform in $\mathbb{R}^d$ we use the convention
\[
(\mathcal{F}f)(\xi)\;\equiv\;\widehat{f}(\xi)\;
=\;\frac{1}{\;(2\pi)^{d/2}}
\int_{\mathbb{R}^d}e^{-\ii x\xi}f(x)\,\ud x\,.
\]
We often write $f\leqslant_{|\,\cdot\,|} g$ when $|f|\leqslant |g|$. 
When not specified otherwise, $C$ denotes a universal positive 
constant and ${1}$ is the 
identity operator on the space that is clear from the context. 
Since the centres $Y$ and the strengths $\a$ will be fixed throughout the 
paper, we shall often omit them from the notation whenever we think no confusion { can occur}.

{ Here is our first main result.} 
\begin{theorem}\lbth(zerolimit) 
Let $\s>1$. Then, 
as a $\Bb_\s$-valued function of $\lam \in \Cb^{+}\setminus\{0\}$, { the resolvent} 
$(H_{Y,\a}-\lam^2)^{-1}$ satisfies the following properties: 
\begin{enumerate}
\item[{\rm (1)}] { The linear map $S{\tD} S$ in $S\C^N$ is 
non-singular if and only if } 
$(H_{Y,\a}-\lam^2)^{-1}$ 
can be extended to a continuous function on $\Cb^{+}$ and 
\begin{align*} 
& (H_{Y,\a}-\lam^2)^{-1}(x,y)\\ 
& = G_0(x-y)-N^{-1}
\big(\la \hat{G}_{0,Y}(x), \hat{\bf 1} \ra + 
\la \hat{\bf 1}, \hat{G}_{0,Y}(y) \ra \big) 
-{N^{-2}} \la \hat{\bf 1}, \tD  \hat{\bf 1} \ra \notag \\ 
& + 
\Big\la [S{\tD}S]^{-1}S \big(\hat{G}_{0,Y}(x)- N^{-1}\tD\hat{\bf 1}\big), 
S \big(\hat{G}_{0,Y}(y)- N^{-1}\tD\hat{\bf 1}\big) \Big\ra + O(g(\lam)^{-1})
\end{align*}
where $O(g(\lam)^{-1})$ satisfies $\|O(g(\lam)^{-1})\|_{\Bb_\s}\leq C|g(\lam)^{-1}|$ 
as $\lam \to 0$.  

\item[{\rm (2)}] Suppose that ${\rm Ker}_{S\C^N}\, S{\tD}S\not=0$ and 
let $T$ be the orthogonal projection in $S\C^N$ onto ${\rm Ker}_{S\C^N}\, S{\tD}S$. Then ${\rm rank}\, T\tD^2 T \leq 1$. Moreover,   
\begin{enumerate}
\item[{\rm (a)}] If $T\tD^2 T$ is non-singular in $T\C^N$, 
then ${\rm rank}\, T = 1$ and, if we write 
$T= \fb \otimes \fb$, $\fb={}^t (f_1, \dots, f_N)$ and 
$\la \fb , {\tD}^2 \fb\ra = \c_0^{-2}$, $\c_0>0$, then 
\bqn  
(H_{Y,\a}-\lam^2)^{-1}  = \c_0^{-2} g(\lam)\ph \otimes \ph + O(1)  \quad (\lam \to 0), 
\lbeq(singular-1)
\eqn
where $\ph(x)\in \R$ satisfies 
$-\lap\ph(x) = \sum_{j=1}^N f_j \d(x-y_j)$ and as $|x| \to \infty$ 
\bqn 
\ph(x) = 
\Big\la \fb, \hat{G}_{0,Y}(x)- \frac{\tD\hat{\bf 1}}{N} \Big\ra 
= b+ \frac{a_1x_1+ a_2 x_2 }{|x|^2} + O\Big(\frac1{|x|^2}\Big)  
\lbeq(s-wave)
\eqn
where $a_1, a_2$ are real constants and   
\bqn 
b = - {N^{-1}}\la \fb, \tD \hat{\bf 1} \ra \not=0 . \lbeq(bnot=0)
\eqn
\item[{\rm (b)}] If $T\tD^2 T$ is singular in $T\C^N$, 
then $T \tD^2 T=0$ and $T\tD= \tD T =0$. 
If $T \Gg_1(Y)T$ is non-singular in $T\C^N$, then   
\begin{align}  
& (H_{Y,\a}-\lam^2)^{-1}(x,y)\notag \\
& = -  (N g\lam^{2})^{-1} \la T \hat{G}_{0,Y}(x), 
[T\Gg_1(Y) T]^{-1} T \hat{G}_{0,Y}(y)\ra +  O(\lam^{-2} ). \notag  \\
 & =  -(Ng\lam^{2})^{-1} \sum_{j=1}^n  \ph_j(x) \ph_j(y)  + O(\lam^{-2} ) \quad (\lam \to 0), 
\lbeq(S2)
\end{align}
where $n= {\rm rank}\, T$, $(a_{j1}, a_{j,2})\not=0$, $j=1, \dots, n$  
are real constants and, as $|x| \to \infty$ 
\bqn
\ph_j(x)= \frac{a_{j1}x_1+ a_{j2} x_2 }{|x|^2} + O(|x|^{-2}). \lbeq(S2aa)
\eqn 
\item[{\rm (c)}] { Both above operators $T\tD^2 T$ and $T \Gg_1(Y)T$ are singular in $T\C^N$ if and only 
if $H_{\alpha,Y}$ has an  eigenvalue at zero. 
More precisely, let} $T_1$ denote the orthogonal projection onto 
${\rm Ker}\, T \Gg_1(Y)T$ and denote by:
\bqn  \label{horia3} 
\Gg_2(Y)= - \left(\frac{\hat\d_{jk}}{8{\pi}N}|y_j-y_k|^2 
\log\Big(\frac{e}{|y_j-y_k|}\Big) \right)\,. 
\eqn 
Then $T_1 {\Gg}_2(Y) T_1$ is non-singular in $T_1\C^N$ and 
\begin{align}  \label{horia10}
& (H_{Y,\a}-\lam^2)^{-1}(x,y)\notag \\
& = -  (N \lam^{2})^{-1} \la T_1 \hat{G}_{0,Y}(x), 
[T_1\Gg_2(Y) T_1]^{-1} T_1 \hat{G}_{0,Y}(y)\ra +  
O(\lam^{-2}g^{-1} ) \\
& = -(N\lam^{2})^{-1} \sum_{j=1}^m \p_j(x) \p_j(y)  
+ O(\lam^{-2}g^{-1} ), \quad \lam \to 0, 
\end{align}
where $m={\rm rank }\, T_1 {\Gg}_2(Y) T_1$ and $\p_1, \dots, \p_m$ are 
zero energy eigenfunctions of $H_{\a,Y}$. 
\end{enumerate}
\end{enumerate}
\end{theorem} 

We say that $H_{\a,Y}$ is of regular type in the case (1) 
and that $H_{\a, Y}$ has 
zero energy resonance of $s$-wave type in the case {\rm (2.a)} 
and $p$-wave type 
in the case {\rm (2.b)}. 
Note that in all cases the leading term as $\lam \to 0$ of $(H_{\a,Y}-\lam^2)^{-1}$ 
is an operator of finite rank. The behaviours of $\ph(x)$ in the $s$-wave resonance 
or $\ph_1(x), \dots, \ph_n(x)$ in the $p$-wave resonance case are similar to the 
corresponding resonance functions of Schr\"odinger operators with regular very 
short range potentials (cf. \cite{JN}, \cite{EG}). 

\vspace{0.2cm}

\noindent
{\bf Remarks.} 

\noindent (1) If $N=1$, $H_{\a,Y}$ is always of regular type. This { directly follows from \refeq(eq:resolvent_identity) where the $\log \lambda$ singularity cancels identically}.

\noindent(2) Let $N=2$. Then: 

(i) $H_{\a,Y}$ is of regular type { if and only if} $\a_1+\a_2 { \neq} \pi^{-1}\log |y_1-y_2|$.

(ii) $H_{\a,Y}$ has a resonance of $s$-wave type { if and only if 
\begin{align*}
&\a_1+\a_2 = \pi^{-1}\log |y_1-y_2|\quad {\rm and}\\
&\left (\alpha_1-(2\pi)^{-1}\log|y_1-y_2|\right )^2+\left (\alpha_2-(2\pi)^{-1}\log|y_1-y_2|\right )^2>0.
\end{align*} }

(iii) $H_{\a,Y}$ has a resonance of $p$-wave type { if and only if} { $$\a_1=\a_2 = (2\pi)^{-1}\log |y_1-y_2|.$$ }

(iv) $H_{\a,Y}$ cannot have a zero energy eigenvalue. 

\noindent(3) { If $N\geq 3$, we shall prove that 
both $T\tD^2 T$ and $T \Gg_1(Y) T$ can be singular and a zero eigenvalue can exist. { A similar} argument 
also applies to the three dimensional case, thus the 
statement on the absence of zero eigenvalue for point 
interactions in \cite{albeverio-solvable} is incorrect. 

More precisely, we have the following result:
\begin{proposition} \label{lemahoria1} Let $N \geq 3$. Assume that 
$\ab= {}^t(a_1, \dots, a_N) \in \R^N\setminus \{0\}$ satisfies 
\bqn \lbeq(horia2) 
\sum_{j=1}^Na_j=0,\;\; \sum_{j=1}^N a_j y_j=0, \;\; {\rm and} \;\; \tD \ab=0.
\eqn
Then the function 
\begin{align}\label{horia6}
\psi(x)=-\sum_{j=1}^N \frac{a_j}{2\pi} \log |x-y_j|
\end{align}
belongs to the domain of $H_{\alpha,Y}$ and $H_{\alpha,Y}\psi=0$. 
Moreover, the converse is also true: any eigenfunction which obeys $H_{\alpha,Y}\psi=0$ 
must be of the form \eqref{horia6} where ${\bf a}\in \R^N\setminus \{0\}$ obeys \refeq(horia2). 
\end{proposition}

For $N=3$ and $y_1, y_2, y_3\in \R^2$ which are collinear or for $N \geq 4$ and arbitrary 
$y_1, \dots, y_N\in \R^2$, there exists $\ab \in \R^N\setminus \{0\}$ which satisfies the first two equations 
of \refeq(horia2). Then we can always find $\alpha$ such that 
$\tilde{D}\ab=0$ and, hence, $H_{\a, Y}$ has an eigenvalue at zero. We will also prove in Lemma \ref{lemahoria2} that a zero mode (if it exists) is always non-degenerate when $N\leq 4$ and we conjecture that this is always true. 
}
\vspace{0.2cm}

The third main result of our paper is the following theorem:
\begin{theorem} \lbth(1)
Suppose that $H_{\a, Y}$ is of regular type. 
Then the wave operators $W_\pm$ are bounded in 
$L^p(\R^2)$ for $1<p<\infty$.   
\end{theorem}

It has been long known (see \cite{DMW}) that the wave operators for one dimensional 
Schr\"odinger operators with point interactions are bounded in $L^p(\R^1)$ 
for all $1<p<\infty$ and, in three dimensions 
{ it was} recently shown (\cite{DMSY}) that they 
are bounded in $L^p(\R^3)$ if and only if $3/2<p<3$. 
Thus, there is a sharp contrast between 
the results in dimensions one or two { compared to dimension} three. { We also note that for Schr\"odinger operators with multiplicative short-range potentials it has been recently proved \cite{EGG} that the wave operators remain bounded in $L^p(\R^2)$ 
for all $1<p<\infty$ even when there is an $s$-wave resonance or an eigenvalue at threshold. }

The intertwining property \refeq(inter) reduces the mapping properties of 
the AC part of the functions $f(H_{\a,Y})$ of $H_{\a, y}$ to 
that of $f(H_0)$ and there is a large body of literature on 
the $L^p$ mapping properties of the wave operators (for this we refer to the 
reference of \cite{DMSY, Ya}). The same intertwining property \refeq(inter) and the well known $L^p$-$L^q$ estimates for 
the free propagator imply the corresponding property of $e^{-itH_{\a,Y}}$. 
We write $\|u\|_p=\|u\|_{L^p(\R^2)}$ for $1\leq p \leq \infty$ and $p'$ is { the} dual 
exponent of $p$ defined by $1/p+1/p'=1$.

\begin{corollary} For any $2\leq p <\infty$, there exists a constant $C_p$ such that 
\bqn \lbeq(p-q)
\|e^{-itH_{\a,Y}}P_{ac}(H_{\a,Y}) u\|_p \leq C_p |t|^{1/p-1/2}\|u\|_{p'}, 
\quad u \in L^p(\R^2) \cap L^2(\R^2).   
\eqn 
\end{corollary}

An immediate corollary of the $L^p$-$L^q$ estimates \refeq(p-q) are the Strichartz estimates 
in two dimensions: We say $(p,r)$ is a $2$-dimensional Strichartz exponent if it satisfies 
\[
\frac1{p}+ \frac{1}{q}=\frac12, \ \ 2<q\leq \infty.
\]
\begin{corollary} Suppose that $H_{\a, Y}$ is of regular type.  Let 
$(p,q)$ and $(s,r)$ be $2$-dimensional Strichartz exponents. Then { there exists a constant $C>0$ such that:}
\begin{gather*}
\left(\int_\R \|e^{-itH_{\a,Y}} u\|_{L^p(\R^2)}^{q}dt \right)^{1/q} \leqslant C \|u\|_2, \\  
\left\| 
\int_0^t e^{-i(t-s)H_{\alpha,Y}}P_{ac}(H_{\alpha,Y}) \,u(s)\,\ud s
\right\|_{L^q(\mathbb{R}_t, L^p(\mathbb{R}^2_x))}
\;\leqslant\;
C\|u\|_{L^{s'}(\mathbb{R}_t, L^{r'}(\mathbb{R}^3_x))}\,.
\end{gather*}
\end{corollary}
For more about Schr\"odinger operators with point interactions we refer to the 
monograph \cite{albeverio-solvable}, while for $L^p$ boundedness of wave operators we refer 
to our previous papers \cite{DMSY,Ya} and references therein. 

 The structure of the remaining text is as follows:
\begin{itemize}
 \item In Section \ref{section2} we give a detailed analysis of the behaviour of $\Ga(\lam)^{-1}$ near $\lambda=0$ and we classify its possible singularities. 
 
 \item In Section \ref{section3} we prove both Theorem \ref{th:zerolimit} and Proposition \ref{lemahoria1}, results which completely characterize the threshold behaviour of the class of zero point interactions we consider here. 
 
 \item Finally, in Section \ref{section4} we give the proof of Theorem \ref{th:1} concerning the $L^p$ boundedness of wave operators. 
\end{itemize}

 \vspace{0.2cm}

\thanks{\noindent {\bf Acknowledgements:} K.Y. acknowledges the support  by JSPS grant in aid for scientific research No. 16K05242. H.C. acknowledges financial support from the Danish Council of Independent Research
| Natural Sciences, Grant DFF–4181-00042.}

\section{The small $\lam$ behaviour of $\Ga(\lam)^{-1}$}\label{section2}

We begin { with the study of} the small $\lam$ behaviour of $\Ga(\lam)^{-1}$.  
In this section, { the notation} $O(g(\lam)^j \lam^k)$, 
$j \in \Z$ and $k \in \N$, represents a scalar or a 
matrix-valued function  
which has { an} asymptotic expansion {when} $\lam \to 0$ as 
\bqn \lbeq(asympto)
O(g(\lam)^j \lam^k)
= C_j g(\lam)^j \lam^k + C_{j-1} g(\lam)^{j-1} \lam^k+ \cdots 
\quad \mbox{mod} \ O(\lam^{k+1}) 
\eqn 
which may be differentiated term by term. For { simplicity},  
{\it we will often omit the variable $\lam$ from various functions and write e.g. 
$g$ for $g(\lam)$, $F$ for $F(\lam)$} and { so on.  The expression $\big(a_{jk}\big)$ will denote an $N \times N$ matrix with entries $a_{jk}$.}  

We shall repeatedly use the following lemma due to Jensen and Nenciu (\cite{JN}) 
in the case when $\Hg$ is finite dimensional. 

\begin{lemma} \lblm(JN) 
Let $A$ be a closed operator in a Hilbert space $\Hg$ 
and $S$ a projection. Suppose $A+S$ has a bounded inverse. 
Then, $A$ has a bounded inverse if and only if 
\[
B= S - S(A+S)^{-1}S 
\]
has a bounded inverse in $S\Hg$ and, in this case, 
\bqn \lbeq(JN-1)
A^{-1}= (A+S)^{-1}+ (A+S)^{-1}SB^{-1}S (A+S)^{-1}.
\eqn 
\end{lemma}

From \refeq(hankel) and \refeq(bessel) and the 
definition \refeq(g) of $g(\lam)$, we  have as 
$\lam \to 0$ that 
\bqn \lbeq(HNKL)
\frac{i}4 H_0^{(1)}(\lam)
= g(\lam) -  \frac1{4}g(\lam)\lam^2  - 
\frac1{8\pi} \lam^2 + O(\lam^4 g(\lam))  
\eqn 
and, for $j\not=k$, 
\begin{multline}
-{\Gg}_\lam(y_j-y_k) = - g(\lam) + \frac1{2\pi}\log |y_j-y_k| \\
+  \frac{\lam^2}{4}g(\lam)|y_j-y_k| ^2 + \frac{\lam^2}{8\pi}|y_j-y_k| ^2\log\Big(\frac{e}{|y_j-y_k|}\Big) 
+ O(\lam^4 g). \lbeq(jkg)
\end{multline}
It follows from \refeq(ga-def) and \refeq(jkg) that as $\lam\to 0$  
\begin{align} 
\Ga(\lam)& =- g(\lam) 
\begin{pmatrix} 1 & \cdots  & 1 \\ \vdots & \ddots & \vdots \\ 
 1 & \cdots  & 1 \end{pmatrix} 
+ \tD + g(\lam)\lam^2 \Big(\frac{\hat\d_{jk}}{4}|y_j-y_k|^2\Big)
\notag \\ 
& \qquad + \frac{\lam^2}{8\pi}\left(
\hat{\d}_{jk} |y_j-y_k| ^2\log\Big(\frac{e}{|y_j-y_k|}\Big) 
\right)_{jk}  + O(\lam^4 g) \notag \\
& =- N g\Big(P - N^{-1}g(\lam)^{-1}\tD + \lam^2 \Gg_1(Y)+ 
\lam^2 g(\lam)^{-1} \Gg_2(Y)+ O(\lam^4) \Big) \notag 
\\ 
& = \colon - N g(\lam) A(\lam),  \lbeq(Gamma-1) 
\end{align} 
where $\tD$ and $\Gg_1(Y)$ are defined in \refeq(deftD) and $\Gg_2(Y)$ in \eqref{horia3}.

We apply \reflm(JN) to { the pair consisting of the operator $A$ appearing in} \refeq(Gamma-1) and 
$S= 1-P$ in the space ${\mathcal H}= \C^N$. 
For { simplicity} we write 
\bqn 
F(\lam)= - N^{-1}g(\lam)^{-1}\tD \,   \lbeq(F) 
\eqn 
so that as $\lam \to 0$ 
\begin{gather*}
A(\lam)+S = 1+ F(\lam) + R_1 + O(\lam^4), \\
R_1=R_1 (\lam, Y)=\lam^2 \Gg_1(Y)+  \lam^2 g(\lam)^{-1} \Gg_2(Y)
\big(=O(\lam^2)\big)\, .
\end{gather*}
{ For small $0 <\lam <\lam_0$, the inverse} $(1+ F)^{-1}$ exists  
and 
\bqn \lbeq(evid)
(1+ F)^{-1}= 1- F + \cdots + (-F)^{n-1} +O(g^{-n}), \quad \lam \to 0  .
\eqn 
{Moreover,} $A(\lam)+S$ is invertible and 
\begin{align}
(A+S)^{-1} & = 
(1 + F)^{-1} \big(1+ (R_1+ O(\lam^4))(1 + F)^{-1}\big)^{-1} \notag \\
& =(1 + F)^{-1} - (1 + F)^{-1}R_1(1 + F)^{-1} + O(\lam^4) \notag \\
& = (1 + F)^{-1} - R_2 + O(\lam^2 {g}(\lam)^{-2}), \lbeq(as) 
\end{align}
where 
\bqn \lbeq(R2)
R_2= \lam^2 \Gg_1(Y) + \lam^2 g^{-1} (N^{-1}\Gg_1(Y)\tD 
+ N^{-1}\tD \Gg_1(Y) + \Gg_2(Y)) \big(=O(\lam^2)\big)  .
\eqn 
From \refeq(as) we have 
\bqn  
B  =  S(1 -(A+S)^{-1})S 
 = S\Big(F(1+F)^{-1}+ R_2 +  O(\lam^2 {g}^{-2}) \Big)S .
 \lbeq(B2)
\eqn 

\subsection{The case {when} $S\tD S$ is invertible in $S\C^N$ }
Suppose that $S\tD S$ is invertible in $S\C^N$. 
Since 
$$
F(1+F)^{-1}= F-F^2+ F^3 - \cdots = -N^{-1}g^{-1}\tD + O(g^{-2}),
$$
by absorbing $R_2 +  O(\lam^2 {g}^{-2})$ of \refeq(B2) 
into $O(g^{-2})$ we have:  
\bqn 
B  = S\Big( -N^{-1}g^{-1}[S{\tD} S]  + O(g^{-2}) \Big)S. 
\lbeq(B2-a)
\eqn
{ Thus, } $B$ is invertible in $S\C^N$ for small $|\lam|>0$ and  
\bqn  \lbeq(binverse)
B^{-1}=  - Ng(\lam) S [S{\tD} S]^{-1}S + SO(1)S.
\eqn   
Combining \refeq(JN-1), \refeq(as) and \refeq(binverse), we see that 
\begin{multline}
A^{-1}= (1+ F)^{-1} \\
-  N g(1+ F)^{-1}S\Big([S\tD S]^{-1}+O(g^{-1})\Big) S(1+ F)^{-1}+ O(g\lam^2).   \lbeq(A-firstcase)
\end{multline}

\begin{lemma}\lblm(5-1) 
$\Gamma(\lam)^{-1}$ is bounded as $\lam \to 0$ 
if and only if $ S{\tD} S $ is non-singular in $S \C^N$. 
In this case,  
\begin{align} 
& \Ga(\lam)^{-1} = -N^{-1}g^{-1} (1+ F)^{-1} \notag \\
& \quad + (1+ F)^{-1}S\Big([S\tD S]^{-1}+
O(g^{-1})\Big)S(1+ F)^{-1}+ O(\lam^2)  \lbeq(G1) \\
& \quad = [S{\tD} S]^{-1}+ O(g^{-1}).  \lbeq(GAM-inverse)
\end{align}
There exists a constant $\lam_0>0$ such that all 
entries of $\Ga^{-1}(\lam)$ 
satisfy 
\bqn \lbeq(8-1)
\pa_\lam^{\ell} [\Ga^{-1}(\lam)]_{jk} \absleq C \lam^{-\ell}, 
\quad \ell=0,1,\dots \quad 0<|\lam|<\lam_0.  
\eqn 
\end{lemma} 

\noindent {\it Proof of the "if" part}. We substitute \refeq(A-firstcase) for $A$ in  
$\Ga^{-1}(\lam) = -N^{-1}g^{-1}A^{-1}$ and pick up the leading (bounded) term. Recall that the 
asymptotic expansion for \refeq(asympto) 
may be differentiated term by term and 
\[
(d/d\lam)^j O(g^{-1})\absleq C \lam^{-j}, \quad j=0,1, \dots.
\]
This proves \refeq(8-1) by virtue of \refeq(GAM-inverse).

{ The proof of the "only if" part of \reflm(5-1) can be completed only when 
 we finish  proving all other lemmas in this section.}

As we shall see in Section \ref{section4}, \refeq(8-1) is sufficient for 
studying the $L^p$ boundedness of the wave operators, however, 
we need { the} more detailed structure \refeq(G1) and 
\refeq(GAM-inverse) in Section \ref{section3} for studying the behavior of 
$(H_{\a,Y}-\lam^2)^{-1}$ as $\lam \to  0$. 

\subsection{The case { when} $S \tD S$ is singular in $S\C^N$}
Suppose that the leading term $S \tD S$ in \refeq(B2-a) is 
singular in $S\C^N$. 
We expand $(1+F)^{-1}$ in \refeq(B2) up to 
the order $O(g^{-3})$ as 
\bqn 
F(1+F)^{-1}=- N^{-1} g^{-1}\tD + N^{-2}g^{-2}\tD^2 - 
N^{-3}g^{-3}R_3, \quad R_3=O(1) , \lbeq(R3)
\eqn 
where the defnition of $R_3=R_3(\lam)$ should be obvious 
(see \refeq(evid)). { We introduce the operator}    
\bqn \lbeq(atD)
B= - N^{-1} g^{-1}
(S{\tD}S -  N^{-1}g^{-1} S{\tD}^2 S  + S(N^{-2}g^{-2}R_3 + Ng R_2) S) 
\eqn 
(recall { from} \refeq(R2) that $R_2=O(\lam^2)$).  { In order to find} the small 
$\lam$ behaviour of $B^{-1}$ in $S\C^N$ we { introduce the pair $(A_1, T)$ where} 
\bqn 
A_1= S{\tD}S +  N^{-1}g^{-1} S{\tD}^2 S  - S(N^{-2}g^{-2}R_3 + Ng R_2) S  \lbeq(A1t)
\eqn
and $T$ is the orthogonal projection in $S\C^N$ onto ${\rm Ker}\, S\tD S$.   
{ Then we again} apply \reflm(JN) to the the pair  $(A_1, T)$. 
Since  
$(S{\tD}S+ T)^{-1}$ is invertible in $S\C^N$ and $A_1= S{\tD}S+ O(g(\lam)^{-1})$,   
then $A_1+ T$ is also invertible for small $0<|\lam|$ and 
\begin{align*} 
& (A_1+ T)^{-1} \notag \\
& = (S{\tD} S + T)^{-1} \big(1+ 
(N^{-1}g^{-1} S{\tD}^2S +SO(g^{-2})S) (S{\tD} S + T)^{-1} \big)^{-1} \notag \\
&= (S{\tD} S + T)^{-1}
- N^{-1}g^{-1} (S{\tD} S + T)^{-1}S{\tD}^2S (S{\tD} S + T)^{-1}+ SO(g^{-2})S. 
\end{align*} 
Since $TS=ST=T$ and $T(S{\tD} S+ T)^{-1}= (S{\tD} S+ T)^{-1}T = T$, 
the operator $B_1$ which corresponds to $B$ when $A_1=A$ satisfies   
\bqn \lbeq(C-1)
B_1 \colon= T - T(A_1+T)^{-1}T 
= N^{-1} g^{-1} T {\tD}^2 T + T O(g^{-2})T .
\eqn 
Here we state the following  lemma:

\begin{lemma} \lblm(td2t) 
The matrix $T {\tD}^2 T$ { has} rank $1$. 
It is singular in $T\C^N$ if and only if 
$\tD T=T\tD=0$ and, if it is non-singular, {then} $\dim T\C^N = 1$. 
\end{lemma} 
\begin{proof} Choose an orthonormal basis 
$\{\eb_2, \dots, \eb_N\}$ of $S\C^N$ 
so that $\{\eb, \eb_2, \dots, \eb_N\}$ is an orthonormal basis of 
$\C^N = \la \eb \ra \oplus S\C^N$, where $\la \eb \ra$ { denotes} the linear span of $\eb$. { Let $K$ denote the matrix of $S{\tD} S$. }Let 
\bqn \lbeq(tgeb)
{\mathcal{M}} = \begin{pmatrix} a & {}^t{\bf a} \\
{\bf a } &   K \end{pmatrix}
\eqn 
be the block matrix representation of $\tD$ in this basis with respect to the decomposition 
$\C^N = \la \eb \ra \oplus S\C^N$. Then
\[
{\mathcal{M}}^2 = \begin{pmatrix} a^2 + |{\bf a}|^2  & a {}^t{\bf a} + {}^t{\bf a} K \\
a {\bf a } + K {\bf a} &  {\bf a}\otimes {\bf a} +  K^2  \end{pmatrix}.
\]
{ We identify $T$ with its matrix.} Since $T$ {projects} onto the kernel 
of $S\tD S$ in $S\C^N$, 
$T {\tD}^2 T$ has the matrix representation $
T{\bf a} \otimes T{\bf a}$
with respect to the basis $\{\eb_2, \dots, \eb_N\}$ { and has}  rank $1$. 
It follows that $T {\tD}^2 T$ is singular in $T\C^N$ if and only if 
$T {\tD}^2 T= (\tD T)^\ast (\tD T)= 0$ or $\tD T = T \tD= 0$ 
and if it is non-singular, it must be that  $\dim T\C^N=1$. 
\end{proof} 

\paragraph{ {\bf (a) The case { when} $T {\tD}^2 T$ is non-singular in $T\C^N= {\rm Ker}\, S\tD S$.} }

 If $T {\tD}^2 T$ is non-singular in $T\C^N$, then $\dim T\C^N=1$ by 
virtue of \reflm(td2t); equation \refeq(C-1) implies that $B_1$ is invertible 
in $T\C^N$ and 
\bqn \lbeq(tbt1)
B_1^{-1}=  N g T [T {\tD}^2 T]^{-1} T + O(1) .
\eqn 
Then, by virtue of \reflm(JN), $A_1$ is invertible in $S\C^N$ and 
\bqn \lbeq(tbt2)
A_1^{-1}= (A_1+ T)^{-1}+(A_1+ T)^{-1}TB_1^{-1}T (A_1+ T)^{-1}.
\eqn 
Since $(A_1+ T)^{-1}$ is bounded as $\lam \to 0$ and 
$(A_1+T)^{-1}T =T(A_1+T)^{-1}= T + O(g^{-1})$, we conclude 
from \refeq(tbt1) and \refeq(tbt2) that in the space $S\C^N$ 
\bqn \lbeq(bnn)
A_1^{-1} =  N gT [T {\tD}^2 T]^{-1} T +  O(1).
\eqn 
Thus, $B= -N^{-1}g^{-1} A_1$ is invertible in $S\C^N$ and  
\bqn \lbeq(B5)
B^{-1}= -Ng A_1^{-1}= -N^2 g^2 T[T {\tD}^2 T]^{-1}T +  O(g) . 
\eqn 
\begin{lemma} \lblm(5-2) Suppose $S{\tD} S$ is singular in 
the space $S\C^N$ 
and let $T$ be the orthogonal projection onto 
the kernel of $S{\tD} S$ in $S\C^N$. Suppose that $T \tD^2 T$ 
is non-singular in $T\C^N$. Then, $T\tD^2 T$ { has} rank $1$ 
and $\Ga(\lam)^{-1}$ has $\log$ singularities 
as $\lam \to 0$. More precisely,  
\begin{align} 
& \Ga(\lam)^{-1} = 
-N^{-1} g^{-1} (1+ F)^{-1} \notag \\ 
& \quad +  
N g(1+ F)^{-1}\Big( T[T\tD^2 T]^{-1}T+ TO(g^{-1})T\Big) (1+ F)^{-1} 
+ O(\lam^2 g) \lbeq(S1) \\
& \quad =  N g[T\tD^2 T]^{-1} +  O(1)  .    \lbeq(S2-extra)
\end{align}
\end{lemma} 

\begin{proof}
We combine \refeq(B5) with \refeq(Gamma-1), \refeq(JN-1) 
and \refeq(as). Recalling that $(A+S)= (1+F)^{-1}+ O(\lam^2)$ 
and $(1+ F)^{-1}= 1+ O(g^{-1})$, we obtain the above result. 
\end{proof}

\paragraph{\bf (b) The case { when} $T {\tD}^2 T$ is singular in $T \C^N $ 
but $T\Gg_1(Y){ T}$ is non-singular} 
If $T {\tD}^2 T$ is singular in $T \C^N $, then, 
by virtue of \reflm(td2t), $T {\tD}= {\tD} T=0$ and 
$R_3 T= T R_3=0$, see \refeq(R3) for the definition of $R_3$. 
Define 
\[
L= S{\tD}S +  N^{-1}g^{-1} S{\tD}^2 S  - SN^{-2}g^{-2}R_3 S, 
\]
so that $TL=LT=0$ and, by virtue of \refeq(A1t), 
\[
A_1= L + SNg R_2 S. 
\]
Recall that $R_2= 
\lam^2 \Gg_1(Y) + \lam^2 g^{-1} (N^{-1}\Gg_1(Y)\tD 
+ N^{-1}\tD \Gg_1(Y) + \Gg_2(Y))$ (see \refeq(R2)).
It follows that in the direct sum decomposition of 
$S\C^N = (1-T)\C^N  \oplus T\C^N$, 
\[
L + T = \begin{pmatrix} L^\perp & 0 \\ 0 &  T \end{pmatrix}= 
L^\perp \oplus T 
\]
where $L^\perp$ is the part of $L$ in $(1-T)\C^N$ and 
\[
(L^\perp)^{-1} =  \{(1-T)S{\tD}S(1-T)\}^{-1} +  O(g^{-1}). 
\]
It follows that 
\begin{align}
& (A_1+ T)^{-1} = (L + T + S Ng R_2 S)^{-1}\notag \\
& = (L^\perp \oplus T)^{-1}  
- (L^\perp \oplus T)^{-1}(S Ng R_2 S) (L^\perp \oplus T)^{-1}  
+ O(\lam^4 g^2 ). \lbeq(byv) 
\end{align}
Define 
\[
R_4(Y)\colon = N^{-1}\Gg_1(Y)\tD + N^{-1}\tD \Gg_1(Y) + \Gg_2(Y).
\]
Then the operator $B_1 = T - T (A_1+ T)^{-1} T $ of 
\refeq(C-1) is given 
as 
\begin{align}  
B_1 & = - T Ng \lam^2 
\Big(\Gg_1(Y) + g^{-1} R_4(Y)+ O(\lam^2 g^2)\Big)T \notag \\
&= - T Ng \lam^2 
\Big(\Gg_1(Y) + g^{-1} \Gg_2(Y) + O(\lam^2 g^2)\Big)T \lbeq(lastB)
\end{align} 
where we used \refeq(R2), the identity 
$T(L^\perp \oplus T)^{-1}={\bf 0}\oplus {1}$ and that 
$\tD T = T \tD=0$ in the final step. 

If $T\Gg_1(Y)T$ is non-singular in $T\C^N$, we see from \refeq(lastB) 
that $B_1$ is invertible in $T\C^N$ for small $0<|\lam|$ and 
\[
B_1^{-1}= - N^{-1}g^{-1}\lam^{-2}T[T\Gg_1(Y) T]^{-1}T+ 
 O(\lam^{-2}g^{-1}) . 
\]
{ We again apply} \reflm(JN) to $A_1$, $B_1$ and $T$ and 
use \refeq(byv) for $(A_1+T)^{-1}$. Then, since 
$(A_1+ T)^{-1}=O(g(\lam))$  as $\lam \to 0$, 
\begin{align*} 
A_1^{-1}& = (A_1+ T)^{-1} - (A_1+ T)^{-1}T B_1^{-1}T (A_1+ T)^{-1}
\\
& = N^{-1}g^{-1}\lam^{-2}
T[T\Gg_1(Y) T]^{-1}T+  O(\lam^{-2}g^{-1}) 
\end{align*}
and  
\[
B^{-1}= -N g(\lam) A_1^{-1}= 
- \lam^{-2}T[T\Gg_1(Y) T]^{-1}T + S O(\lam^{-2}g^{-1})S. 
\]
By using \refeq(as), we have proved the following lemma: 

\begin{lemma} \lblm(5-4) Suppose that $S{\tD} S$ is singular in the space $S\C^N$ 
and $T \tD^2 T=0$. Suppose further that $T \Gg_1(Y)T$ is non-singular in $T\C^N$. 
Then, as $\lam \to 0$ { we have:}
\begin{align} 
& { \Ga(\lam)^{-1}}= 
- N^{-1} g^{-1} (1+ F)^{-1} +  N^{-1} g^{-1}\lam^{-2} \notag \\
& \times (1+ F)^{-1}S \Big( T[T\Gg_1(Y) T]^{-1}T+ 
O(g^{-1})\Big) S(1+ F)^{-1} 
+ O(g^{-1}) \lbeq(S1a) \\
&= N^{-1} g^{-1}\lam^{-2}T[T\Gg_1 (Y) T]^{-1}T + O(\lam^{-2}) ,   
\lbeq(S2a)
\end{align}
where we wrote $ST= T$ for simplicity in \refeq(S2a). 
\end{lemma}

\paragraph{\bf (c) The case {when both} $T\tD^2 T$ and $T\Gg_1(Y)T$ are 
 singular in $T \C^N $} 
 
 We note the identity 
\bqn \lbeq(tgg2)
\Gg_2(Y) = \frac1{2\pi} \Gg_1(Y) + \tilde{\Gg}_2(Y), \; {\rm where}\;  
\tilde{\Gg}_2(Y) = \Big(\frac{\hat{\d}_{jk}} {8\pi N}
|y_j-y_k|^2 \log |y_j-y_k| \Big). 
\eqn 

\begin{lemma} \lblm(gGY2) Suppose { that both} $T\tD^2 T$ and $T\Gg_1(Y)T$ 
are singular. Let $T_1$ be the orthogonal projection onto 
${\rm Ker}\, T \Gg_1(Y)T$ in $T \C^N $. 
Then both $T_1 \Gg_2(Y)T_1$ and $T_1 \tilde{\Gg}_2(Y)T_1 $ are 
non-singular in $T_1\C^N$. 
\end{lemma} 
\begin{proof} Define the $N \times N$ matrix $M$ by 
\[
M= \big(\hat{\d}_{jk}|y_j-y_k|^2 \log(|y_j-y_k|^2)\big).
\] 
Due to the presence of $T_1$, it suffices to show that  
$T_1 M T_1$ is non-singular in $T_1\C^N$. 
Because all the matrices we worked with until now were 
real and symmetric, 
we may choose their eigenvectors to be real. 
Thus the matrices of $T_1$ and of 
$T_1MT_1$ are also real and self-adjoint. Hence we can 
choose the eigenvectors of 
$T_1MT_1$ to be real. Let ${f}=T_1 f$ be a normalized real 
eigenvector of $T_1MT_1$ associated with the smallest eigenvalue. 
We show that necessarily $\la f,Mf \ra >0$, hence 
$T_1MT_1$ is positive definite and non-singular on the range 
of $T_1$. 

For $\fb={}^t(f_1, \dots, f_N) \in T_1\C^N$, define the function  
\[
F(\lambda)=\sum_{1\leq j\not = k \leq N} 
f_jf_k|y_j-y_k|^2 \log(|y_j-y_k|^2+\lambda),\quad \lambda\geq 0.
\]
We want to show that $F(0)>0$. We observe that $F(\lam)$ is smooth for $\lam\geq 0$ 
and that $\lim_{\lambda\to\infty}F(\lambda)=0$ because 
$\sum_{1\leq j\not = k \leq N} f_jf_k|y_j-y_k|^2 = \la \fb, T_1 \Gg_1(Y) T_1 \fb\ra=0 $ 
for $\fb\in T_1\C^n$ and 
\[
F(\lambda)=\sum_{1\leq j\not = k \leq N} 
f_jf_k|y_j-y_k|^2 \big(\log(|y_j-y_k|^2+\lambda) - \log \lam\big). 
\]
{ We will prove that 
$F'(\lambda)< 0 $ for all 
$\lambda >0$, which implies  
$F(0)>\lim_{\lambda\to\infty}F(\lambda)=0$ and finishes the proof. In order to do that, we compute}
$$F'(\lambda)=\sum_{j,k=1}^N f_jf_k\frac{|y_j-y_k|^2}{ |y_j-y_k|^2+\lambda}
=-\sum_{j,k=1}^N f_jf_k\frac{\lambda}{ |y_j-y_k|^2+\lambda}.$$
For $t>0$ we have the identity:
\begin{align*}
e^{-t(|y_j-y_k|^2+\lambda)}=\frac{e^{-\lambda t}}{4\pi t}\int_{\R^2}e^{ip(y_j-y_k)}e^{-p^2/(4t)}dp,
\end{align*}
and also: 
$$\frac{1}{ |y_j-y_k|^2+\lambda}=\lim_{n\to\infty}\int_{n^{-1}}^n e^{-t(|y_j-y_k|^2+\lambda)}dt.$$
Thus, 
$$F'(\lambda)=-\lambda \lim_{n\to\infty} \int_{n^{-1}}^n dt\; \frac{e^{-\lambda t}}{4\pi t}\int_{\R^2}dp\; e^{-p^2/(4t)}\left |\sum_{j=1}^N e^{ip\cdot y_j}f_j\right |^2 \leq 0.$$ 
{ The above inequality is in fact strict for every  $\lam>0$, since  $F'(\lambda)=0$ for some $\lam>0$ would imply} 
$\sum_{j=1}^N e^{ip\cdot y_j}f_j=0$ for all $p\in \R^2$, { which is equivalent with} $f_1= \dots = f_N=0$. 
\end{proof}

For studying $B_1^{-1}$ of \refeq(lastB), we let $A_2$ be the linear operator in 
$T\C^N$ inside the parenthesis 
of \refeq(lastB): 
\bqn 
A_2= T\big( \Gg_1(Y) + g^{-1} \Gg_2 (Y)+ O(\lam^2 g^2)\big)T, 
\eqn 
and apply \reflm(JN) once again to the pair $(A_2, T_1)$. 

{ The inverse} $(A_2+ T_1)^{-1}$ exists in $T\C^N$ for small 
$0<|\lam|$ and, omitting the variable $Y$, 
\begin{multline*}
(A_2+ T_1)^{-1}= (T\Gg_1T+ T_1)^{-1} \\ - 
g^{-1} (T\Gg_1 T+ T_1)^{-1} T\Gg_2(Y) T(T\Gg_1T+ T_1)^{-1} + O(g^{-2}).
\end{multline*}
We need to consider the invertibility of 
\[
B_2= T_1 - T_1 (A_2+ T_1)^{-1} T_1 .
\]
Since $T_1 T\Gg_1 T= T\Gg_1 TT_1=0$ 
and $T_1 (T\Gg_1 T+ T_1)^{-1}= (T\Gg_1 T+ T_1)^{-1}T_1 = T_1$, 
we have 
$B_2 = - g^{-1}T_1 \tilde{\Gg}_2(Y) T_1 + O(g^{-2})$.  
Because $T_1 \tilde{\Gg}_2(Y) T_1$ is non-singular in $T_1 \C^N$ 
by virtue  of \reflm(gGY2), $B_2^{-1}$ exists for small $|\lambda|>0$ and 
\[ 
B_2^{-1}= -g T_1 [T_1 \tilde{\Gg}_2(Y) T_1]^{-1}T_1+ O(1).
\] 
Then, by virtue of \reflm(JN), $A_2^{-1}$ also exists for small $|\lambda|>0$ and 
\begin{align*}
A_2^{-1} 
& = (A_2+ T_1)^{-1}- (A_2+ T_1)^{-1}T_1 B_2^{-1}T_1(A_2+ T_1)^{-1} 
\\
& = g T_1 [T_1 \tilde{\Gg}_2(Y) T_1]^{-1}T_1+ O(1)
\end{align*}
where we used $(A_2+ T_1)^{-1}=O(1)$ and 
$(A_2+ T_1)^{-1}T_1= T_1(A_2+ T_1)^{-1}= T_1 + O(g^{-1})$ 
in the second step.  Thus, we have 
\begin{align}  
B_1^{-1} & = - T^{-1} N^{-1}g(\lam)^{-1}A_2^{-1} \notag \\
& = - T^{-1} N^{-1} \lam^{-2} 
T_1 [T_1 \tilde{\Gg}_2(Y) T_1]^{-1}T_1 T 
+ O(\lam^{-2}g(\lam)^{-1}). \lbeq(lastB-a)
\end{align} 
Then, exactly as in the case (b), we have 
\[
A_1^{-1}= T^{-1} N^{-1} \lam^{-2} 
T_1 [T_1 \tilde{\Gg}_2(Y) T_1]^{-1}T_1 T 
+ O(\lam^{-2}g(\lam)^{-1})
\]
and 
\[
B^{-1}= -N g(\lam)A_1^{-1}= T^{-1} N^{-1} \lam^{-2} 
T_1 [T_1 \tilde{\Gg}_2(Y) T_1]^{-1}T_1 T 
+ O(\lam^{-2}g(\lam)^{-1}).
\]
Repeating the argument in the last part of the proof 
of \reflm(5-1), we prove the following lemma: 

\begin{lemma} \lblm(5-5) 
Suppose that $T {\tD}^2 T=0$ and $T \Gg_1 (Y)T$ are both singular in 
$T \C^N $. Let $T_1$ be the projection to 
${\rm Ker}\, T \Gg_1 (Y)T$ in $T \C^N $. 
Then $T_1 \tilde{\Gg}_2(Y)T_1$ is non-singular 
in $T_1\C^N$ and, as $\lam \to 0$ we have: 
\begin{align} 
{ \Ga(\lam)^{-1}} & = 
- N^{-1} g^{-1} (1+ F)^{-1} -  N^{-1} \lam^{-2} (1+ F)^{-1} \notag \\
& \times S \Big( T_1 [T_1 \tilde{\Gg}_2(Y) T_1]^{-1}T_1 + 
O(g^{-1})\Big)S (1+ F)^{-1} 
+ O(1) \lbeq(S1b) \\
& = - N \lam^{-2}T_1[T_1 \tilde{\Gg}_2(Y) T_1]^{-1}T_1 + 
O(g^{-1} \lam^{-2})   \lbeq(S2b)
\end{align}
where in \refeq(S2b) the first $T_1$ should be considered as a linear map from 
$T_1\C^N$ into $\C^N$ and the last {one} from $\C^N$ into $T_1\C^N$. 
\end{lemma}

\section{Proof of Theorem \ref{th:zerolimit} and Proposition \ref{lemahoria1}} \label{section3}
 
{ We start with Proposition \ref{lemahoria1} because it is strongly related to the part (2)(c) of Theorem \ref{th:zerolimit} and it completely characterizes the zero modes of $H_{\alpha,Y}$.} 
 
\subsection{Proof of Proposition \ref{lemahoria1}} 
We first show that $\psi$ in \eqref{horia6} { belongs to the domain of $H_{\a, Y}$} and  satisfies $H_{\a, Y}\p=0$. It is obvious that 
$\psi \in L^2_{\rm loc}(\R^2)$. For large $|x|$  we have 
\[
\log |x-y_j|=\log|x|-\frac{x\cdot y_j}{|x|^2}+\mathcal{O}(|x|^{-2})
\]
and \refeq(horia2) implies that $\psi(x)$ behaves like $|x|^{-2}$ at infinity hence 
it is square integrable. We next show $\psi\in D(H_{\a, Y})$. Let $\mu\in \mathbb{C}^{+}$ be 
such that $\Gamma_{\alpha,Y}(\mu)$ is invertible and define the vector 
\[
v_\mu(x):=\psi(x)- \sum_{j=1}^N a_j
\mathcal{G}_{\mu}(x-y_j)=\sum_{j=1}^N a_j \left (-\frac{1}{2\pi}\log |x-y_j|-
\mathcal{G}_{\mu}(x-y_j)\right ).
\]
Clearly, $v_\mu \in H^2(\R^2)$ because the logarithmic singularities of $\psi$ are removed. 
Moreover, we have:
\begin{align*}
v_\mu(y_k)&=\sum_{j=1}^Na_j\left (\hat{\delta}_{jk}(-(2\pi)^{-1}\log |y_k-y_j| -
\mathcal{G}_{\mu}(y_k-y_j))+\frac{\delta_{jk}}{2\pi}(\log(\mu/2i)+\gamma)\right )\\
&=\sum_{j=1}^N[\Gamma_{\alpha,Y}(\mu)]_{kj}a_j,
\end{align*}
where we used the fact that $\tilde{D}\ab=0$. Thus we have:
$$\psi(x)=v_\mu(x)+\sum_{j,k=1}^N [\Gamma_{\alpha,Y}(\mu)]^{-1}_{jk}v_\mu(y_k) \mathcal{G}_\mu(x-y_j)$$
which (see \eqref{eqn:dom}) shows that $\psi$ belongs to the domain of $H_{\alpha,Y}$. By computing the distributional Laplacian of $v_\mu$ we obtain:
\[
(-\Delta -\mu^2)v_\mu=\mu^2\sum_{j=1}^N \frac{a_j}{2\pi}\log|x-y_j|
=-\mu^2\psi=(H_{\alpha,Y}-\mu^2)\psi,
\]
which confirms that $H_{\alpha,Y}\psi=0$.

{
We now prove the converse. Assume that $\psi$ is in the domain of $H_{\alpha,Y}$ 
and $H_{\alpha,Y}\psi=0$.  Let $\mu$ be such that $\Gamma_{\alpha,Y}(\mu)$ is invertible, 
viz. $\mu\in \C^{+}\setminus \Eg$, $\Eg\subset i[0,\infty)$ being the 
square roots of negative eigenvalues of $H_{\a,Y}$. Then there must exist 
a function $v_\m \in H^2(\R^2)$ such that 
\begin{align}\label{horia20}
\psi(x)=v_\mu(x)+\sum_{j,k=1}^N [\Gamma_{\alpha,Y}(\m)]^{-1}_{jk}v_\m(y_k) \mathcal{G}_\m(x-y_j).
\end{align}
Define $\ab= {}^t(a_1, \dots, a_N)$ by 
\[
a_j = \sum_{k=1}^N [\Gamma_{\alpha,Y}(\m)]^{-1}_{jk}v_\m(y_k), \quad j=1, \dots, N. 
\]
The vector $\ab$ { must be} independent of $\m$ because { all its components can be directly expressed  in terms of $\psi$ by using \refeq(G0g) in \eqref{horia20}:} 
\[
a_j=-2\pi \lim_{x \to y_j} \psi(x)(\log|x-y_j|)^{-1}, \quad j=1, \dots, N.
\]
Since $(H_{\alpha,Y}-\mu^2)\psi=(-\Delta-\mu^2)v_\mu$, we have the equation 
\[
-\Delta v_\mu(x)=-\mu^2\sum_{j}^N a_j \mathcal{G}_\mu(x-y_j).
\]
{ In momentum coordinates} we have: 
\[
\hat{v}_\mu(p)=-\frac{\mu^2}{2\pi}\sum_{j=1}^N a_j \frac{e^{-ip\cdot y_j}}{p^2(p^2-\mu^2)}.
\]
Since $v_\mu \in L^2(\R^2) \setminus \{0\}$, $\ab$ must obey the first two equations of \refeq(horia2). 

We now show that if we take $\m= i\lam$ with sufficiently small $\lam>0$ and $\mu \to 0$, 
then 
$\|v_{\mu}\|_{H^2} \to 0$, hence $v_\m(x) \to 0$ uniformly. Since $\ab$ satisfies 
the first two equations of \refeq(horia2) this would imply 
\[
\lim_{\m\to 0} \Gamma_{\alpha,Y}(\m)\ab = \tD \ab =0  
\]
and the { desired identity}
\[
\psi(x)=- \frac1{2\pi} \sum_{j=1}^N a_j \log |x-y_j|. 
\]
To { show that} $\|v_{\mu}\|_{H^2} \to 0$ we first observe the trivial estimate  
\[
\|\lap v_{i\lam}\|= \|\lam^2 \lap (-\lap + \lam^2)^{-1} \p \| \leq \lam^2 \|\p\|.
\]
{ In momentum coordinates} we have that 
\[
\hat{v}_{i\lam}(p) = \lam^2 (p^2+\lam^2)^{-1}\hat{\p}(p) 
\absleq \hat{\p}(p) \ \ \mbox{and} \ \ \lim_{\lam\to 0} \hat{v}_{i\lam}(p)= 0 \ \ 
\mbox{if} \ \ p\not=0.  
\] 
It follows by the dominated convergence theorem, 
\[
\|v_{i\lam}\|_{L^2(\R^2)} \to 0 \quad (\lam \to 0).
\]
Hence $\|v_{i\lam}\|_{H^2} \to 0$ as $\lam \to 0$. This finishes the proof of Proposition \ref{lemahoria1}.
}

{ 
\begin{lemma}\label{lemahoria2}
Let $N\leq 4$ and assume that $H_{\alpha,Y}$ has an embedded eigenvalue at zero. Then the eigenvalue is non-degenerate. 
\end{lemma}
\begin{proof}
We know that we need at least $N=3$ in order to have a zero mode. Without any loss of generality, up to a translation, a scaling and a relabelling, we may always assume that $y_1=0$, $|y_2|=1$, and $1\leq |y_j|$ for all $j\geq 3$. 

If $N=3$ then $y_2$ and $y_3$ must be linearly dependent, otherwise the first two constraints of \eqref{eqn:horia2} impose ${\bf a}=0$ and no zero mode can exist. If $y_2$ and $y_3$ are collinear then up to a translation, a scaling and a relabelling we may assume that $y_2$ and $y_3$ have the same direction and $|y_1|=0<|y_2|=1<|y_3|$. We write $a_2=-|y_3| a_3$ and $a_1=-a_2-a_3=(|y_3|-1)a_3$. Thus all the compatible ${\bf a}$'s belong to an one-dimensional subspace generated by the vector with components ${\bf a}_1={}^t(|y_3|-1,-|y_3|,1)$. Now from the equation $\tilde{D}{\bf a}_1=0$ we can find the right combination of $\alpha$'s (uniquely determined by $|y_3|$) for which a zero mode can exist. Thus if a zero mode exists, it must be non-degenerate. 

Now let $N=4$. We know that $y_2$, $y_3$ and $y_4$ are linearly dependent. There are two possibilities: either these three vectors are all collinear or they are not. 

If they are not collinear, for example $y_2$ and $y_3$ are linearly independent, then given any $a_4\in \R$ we may uniquely determine $a_2$ and $a_3$ from the equation $a_2y_2+a_3y_3=-a_4 y_4$ and also $a_1=-a_2-a_3-a_4$. Thus we are again in a situation in which the compatible ${\bf a}$'s form an one-dimensional family. As in the $N=3$ case, if a zero mode exists, it must be unique.

Let us now assume that all four points are collinear. We may also assume without loss of generality that $y_2$, $y_3$ and $y_4$ have the same direction and 
$$|y_1|=0<|y_2|=1<|y_3|<|y_4|.$$ Then we have $a_2=-a_3|y_3|-a_4 |y_4|$ and $a_1=(|y_3|-1)a_3+(|y_4|-1)a_4$. This time, the family of compatible ${\bf a}$'s is two dimensional, generated by the following two linearly independent vectors: 
$${\bf a}_1={}^t(|y_3|-1, -|y_3|, 1,0)\; {\rm and}\; {\bf a}_2={}^t(|y_4|-1, -|y_4|,0, 1).$$
To each generator we can separately find some $\alpha$'s for which a zero mode would exist, but we want to see if we can find one joint $\alpha$ for which 
both equations $\tilde{D}{\bf a}_1=0$ and $\tilde{D}{\bf a}_2=0$ are simultaneously satisfied. By solving for $\alpha$ in both equations we obtain four compatibility relations involving $|y_3|$ and $|y_4|$. The one involving $\alpha_1$ imposes the condition: 
$$\frac{\log |y_3|}{|y_3|-1}=\frac{\log |y_4|}{|y_4|-1}.$$
But the function $(\log t)/(t-1)$ is strictly decreasing if $t\in (1,\infty)$, hence the above equality cannot hold true. Thus the zero mode is unique if it exists. 
\end{proof}

\noindent {\bf Remark}. If $N\geq 5$, the family of ${\bf a}$'s which are compatible with the first two equations in \eqref{eqn:horia2} is always at least two dimensional. The compatibility relations (only involving the $y$'s) which are obtained from the condition that the $\alpha$'s must be the same, are much more complicated. Nevertheless, they can always be written as an equation of the type $F(y_3,...,y_N)=0$ where $F:\R^{N-2}\mapsto \R^N$ is a rather complicated function; here $y_1=0$ and $y_2={}^t(1,0)$ are fixed and no two $y$'s can coincide. We conjecture that no degeneracy is possible when $N\geq 5$.   
}
\subsection{{ Proof of Theorem \ref{th:zerolimit}: preliminaries}}

We will study the operator 
\bqn 
D(\lam)=(H_{\a,Y}-\lam^2)^{-1}- (H_{0}-\lam^2)^{-1}
=\la \widehat{\Gg}_{\lam,Y}(x), \Gamma_{\alpha,Y}(\lam)^{-1} \widehat{\Gg}_{\lam,Y}(y) \ra 
\lbeq(prob-1a)
\eqn 
when $\lam \in \Cb^{+}\setminus\{0\} $ converges to zero, by using the results of Section \ref{section2}. { Our results} will be stated 
for $\lam>0$, however, they hold for 
$\lam \in \Cb^{+}\setminus\{0\}$ with the same proof. 
As before, we identify operators with their integral kernels.

We define 
\[
R_0(\lam, x)= 
\Gg_\lam (x)- g(\lam |x|) = \Gg_\lam (x)- g(\lam) - G_0(x)  
\]
and use the vector notation  
\[
\widehat{R}_{0,Y}(\lam, x)=
\begin{pmatrix} {R}_{0}(\lam ,x-y_1) \\ 
\vdots \\ {R}_{0}(\lam, x-y_N) \end{pmatrix}, \qquad 
\widehat{g}_{\lam} = g(\lam) \hat{\bf 1} 
\]
so that 
\bqn \lbeq(64)
\widehat{\Gg}_{\lam,Y} (x)=\hat{g}(\lam) + \widehat{G}_{0,Y}(x) +\widehat{R}_{0,Y}(\lam, x). 
\eqn

By virtue of \refeq(HNKL) for the Hankel function for small $\lam$, we have for any 
constant $C_1>0$ and for an arbitrary small $0<\d$ that  
\bqn 
R_0(\lam, x) \absleq \, C_\d |\lam |x||^{\d}, \quad  |\lam |x||< C_1, 
\lbeq(Hankel-small) 
\eqn 
and from \refeq(large-hankel) for large $\lam$ that 
\bqn 
\Gg_\lam (x)\, \absleq\, C |\lam |x||^{-1/2}, \quad  |\lam |x|| \geq C_1. \lbeq(Hankel-large)
\eqn 

We take a cut-off function $\chi\in C_0^\infty(\R^2)$ such that 
\[
\chi(x)=1 , \ \mbox{for} \ |x|\leq 1 \ \ \mbox{and} \ \ \chi(x)=0, \ \mbox{for} \ |x|\geq 2 
\]
and define for $\lam >0$ 
\[
\chi_\lam (x) = \chi(\lam x), \quad 
\widehat{\Gg}_{\lam,Y}^{\leq} (x) = \chi_\lam (x) \widehat{\Gg}_{\lam,Y}, \ \
\widehat{\Gg}_{\lam,Y}^{\geq} (x) = (1-\chi_\lam (x)) \widehat{\Gg}_{\lam,Y}, \ 
\]
and likewise for other functions. 
{ To shorten the} formulas, we often omit the variables from various functions. 

\begin{lemma} For any $\lam_0>0$ and $\s>1$, there exists $C>0$ such that 
the following estimates are satisfied for $0<\lam<\lam_0$: 
\begin{gather} 
\|\widehat{\Gg}_{\lam,Y}^{\geq} \|_{L^2_{-\s}} \leq  C \lam^{\s-1}, \quad 
\|\widehat{G}_{0,Y}^{\geq}\|_{L^2_{-\s}} \leq C \lam^{\s-1} \la g(\lam) \ra.  \lbeq(2nd) \\
\|\widehat{\Gg}_{\lam,Y}^{\leq} \|_{L^2_{-\s}} \leq  C \la g(\lam)\ra,  \quad 
\|\widehat{G}_{0,Y}^{\leq}\|_{L^2_{-\s}} \leq C  .  \lbeq(R-3) 
\end{gather}
For any $0<\d<\s-1$, there exists $C>0$ such that for $0<\lam<\lam_0$
\bqn 
\|\widehat{R}_{0,Y}(\lam,x)\|_{L^2_{-\s}}\leq C \lam^{\d}. \lbeq(R-0Y) 
\eqn 
\end{lemma} 
\begin{proof} 
By virtue of \refeq(Hankel-large), we have for $\s>1$ and for small 
$0<\lam <\lam_0$ that 
\begin{gather*}
\|\widehat{\Gg}_{\lam,Y}^{\geq} \|_{L^2_{-\s}} 
\leq C \lam^{-\frac12} \Big(
\int_{|x|\geq C\lam^{-1}} |x|^{-1-2\s}dx\Big)^{1/2} 
= C \lam^{\s-1} .
 \\
\|\widehat{G}_{0,Y}^{\geq}\|_{L^2_{-\s}} 
\leq C \Big(\int_{|x|\geq C\lam^{-1}} 
\frac{(\log |x|)^2}{|x|^{2\s}}dx\Big)^{1/2} 
= C \lam^{\s-1} \log \lam .
\end{gather*}
This proves \refeq(2nd). The first of the following estimates is obvious and the 
second follows from \refeq(Hankel-small):
\begin{gather*} 
\Big( 
\int_{|x|\leq \lam^{-1}} 
|\widehat{G}_{0,Y}(x)|^2 \ax^{-2\s} dx 
\Big)^{1/2} \leq C <\infty, \\
\Big(
\int_{|x|\leq \lam^{-1}} 
|\widehat{\Gg}_{\lam,Y}(x)|^2 \ax^{-2\s} dx  
\Big)^{1/2}
\leq C \la g(\lam)\ra. 
\end{gather*}
{ This yields} \refeq(R-3). 
By virtue of \refeq(Hankel-small), we have for any $0<\d<\s-1$ that 
\begin{align}
& \Big(\int_{\R^2} |\widehat{R}_{0,Y}^{\leq}(\lam, x)|^2 
\ax^{-2\s} dx\Big)^{1/2}
\notag \\
& \hspace{1cm} 
\leq C \lam^\d \Big(\int_{|x|\leq \lam^{-1}} |x|^{2\d} 
\ax^{-2\s} dx  \Big)^{1/2} \leq C \lam^{\d} .
\lbeq(R-5)
\end{align}
Estimate \refeq(2nd) implies 
$\|\widehat{R}_{0,Y}^{\geq}(\lam)\|_{L^2_{-\s}}\leq C \lam^{\d}$
for any $0<\d<\s-1$, which completes the proof of  the lemma. 
\end{proof}

We will now study $D(\lam)$ of \refeq(prob-1a) in the space 
$\Bb_\s= \Bb(L^2_\s(\R^2), L^2_{-\s}(\R^2))$, $\s>1$. 
 Note that $L^2_\s(\R^2) \subset L^1(\R^2)$ when $\s>1$. 

We begin with studying the contribution to $D(\lam)$ 
of $-N^{-1}g^{-1}(1+ F)^{-1}$ which is the common first term 
in the right hand sides of the first formulas for 
$\Gamma(\lam)$ in \reflmsss(5-1,5-2,5-4,5-5). 

\begin{lemma}
\lblm(first-term)  Let $\s>1$. Then, as a $\Bb_\s$-valued function 
of $\lam>0$ we have  
\begin{align} \lbeq(Kern-1)
& - N^{-1}g^{-1} 
\la \widehat{\Gg}_{\lam, Y} (x) ,(1+ F)^{-1} \widehat{\Gg}_{\lam, Y} (y) \ra \\
& \quad =  -g - N^{-1}
\big(\la \hat{G}_{0,Y}(x), \hat{\bf 1} \ra 
+ \la \hat{\bf 1}, \hat{G}_{0,Y}(y) \ra 
+{N^{-1}} \la \hat{\bf 1}, \tD  \hat{\bf 1} \ra \big)  +O(g^{-1}), 
\lbeq(Kern-1a)
\end{align} 
where $O(g^{-1})$ is such that 
$\|O(g^{-1})\|_{\Bb_\s} \leq C|g(\lam)|^{-1}$ as $\lam \to 0$. 
\end{lemma} 
\begin{proof} 
We substitute $\widehat{\Gg}_{\lam, Y} (x)
=\widehat{\Gg}_{\lam, Y}^{\geq} (x)+ \widehat{\Gg}_{\lam, Y} ^{\leq} (x)$ 
and likewise for $\widehat{\Gg}_{\lam, Y} (y)$ in \refeq(Kern-1). 
Then, \refeq(2nd) and \refeq(R-3) imply that as $\lam \to 0$ 
\bqn \lbeq(Kern-2)
\refeq(Kern-1) = - 
N^{-1}g^{-1} \la \widehat{\Gg}_{\lam, Y}^{\leq} (x) ,
(1+ F)^{-1} \widehat{\Gg}_{\lam, Y}^{\leq} (y) \ra +  O(\lam^{\d}) 
\eqn 
for any $0<\d<\s-1$. Multiplying \refeq(Hankel-small) 
by $\chi_\lam(x)$ we have    
\[
\widehat{\Gg}_{\lam,Y}^{\leq}(x)= 
\chi_\lam(x)\hat{g}(\lam) + \widehat{G}_{0,Y}^{\leq}(x) 
+ \widehat{R}_{0,Y}^{\leq}(\lam,x), 
\]
and likewise for $\widehat{\Gg}_{\lam,Y}^{\leq}(y)$ 
which we insert in the right of \refeq(Kern-2). 
This produces nine terms out of which five contain 
$\widehat{R}_{0,Y}^{\leq}(\lam,x)$ or 
$\widehat{R}_{0,Y}^{\leq}(\lam,y)$ and, 
by virtue of \refeq(R-3) and \refeq(R-0Y), they are  
bounded by $C \lam^{\d}$, $\d<\s-1$ in $\Bb_\s$. 
We { collect them} into $O(g^{-1})$ of \refeq(Kern-1a). 
Moreover, we trivially have 
\[
\| N^{-1}g(\lam)^{-1} \la \widehat{G}_{0,Y}^{\leq} (x) ,
(1+ F(\lam))^{-1}\widehat{G}_{0,Y}^{\leq} (y) \ra \|_{\Bb_\s} \leq C \la g(\lam)\ra^{-1}
\]
{ and we include this too into} $O(g^{-1})$.  
Thus, we only have the following three terms 
$Z_1, Z_2$ and $Z_3$ to deal with. 
\begin{align} 
Z_1 & = - N^{-1}g^{-1} \la  \chi_\lam(x)\hat{g} ,
(1+ F)^{-1} \chi_\lam(y)\hat{g} \ra,  \lbeq(z1) 
\\
Z_2 & = - N^{-1}g^{-1} \la  \widehat{G}_{0,Y}^{\leq}(x) ,
(1+ F)^{-1} \chi_\lam(y)\hat{g} \ra, \lbeq(z2)
\\
Z_3 & =  - N^{-1}g^{-1} \la  \chi_\lam(x)\hat{g}, 
(1+ F)^{-1}\widehat{G}_{0,Y}^{\leq}(y)  \ra . \lbeq(z3)
\end{align}
We have by using that 
$\|1 -\chi_\lam \|_{L^2_{-\s}} \leq C \lam^{\s-1}$ for $0<\lam<C_1$  
\begin{align}
Z_1 & = - N^{-1} g^{-1}\la \hat{g}, 
(1+ F)^{-1} \hat{g}\ra \chi_\lam \otimes \chi_\lam \notag \\
& = - N^{-1} g^{-1}
\Big(\la \hat{g}, \hat{g}\ra - \la \hat{g}, F \hat{g}\ra + 
\la \hat{g}, F^2 (1+ F)^{-1} \hat{g}\ra \Big) \chi_\lam 
\otimes \chi_\lam \notag  \\
& = \left(- g-  N^{-2}\la \hat{\bf 1}, \tD \hat{\bf 1}\ra \right)
(1 \otimes 1) + O(g^{-1}). 
\lbeq(fin-1)  
\end{align} 
In a similar fashion we have 
\begin{align}
Z_2 & = - N^{-1} \la  \widehat{G}_{0,Y}(x), \hat{\bf 1} \ra + O(g^{-1}), 
\lbeq(fin-2) \\
Z_3 & = - N^{-1} \la \hat{\bf 1},  \widehat{G}_{0,Y}(y)\ra + O(g^{-1}).  \lbeq(fin-3)
\end{align}
{ The combination} of \refeq(fin-1), \refeq(fin-2) and \refeq(fin-3) concludes 
the proof of \reflm(first-term).
\end{proof}

The following corollary shows that the sum of the first term and 
the contribution by the common first term $-N^{-1}g^{-1}(1+ F)^{-1}$ of $\Ga(\lam)$ in the Lemmas \ref{lm:5-1}, \ref{lm:5-2}, \ref{lm:5-4}, and \ref{lm:5-5} 
 to the second term on the right of  
\bqn \lbeq(reda)
(H_{\a,Y}-\lam^2)^{-1}= \Gg_\lam(x-y)+ 
\la \widehat{\Gg}_{\lam,Y}(x), \Ga(\lam)^{-1} \widehat{\Gg}_{\lam,Y}(y) \ra 
\eqn 
is bounded in $\Bb_{\s}$: 

\begin{corollary} \lbcor(first-term)
Let $\s>1$. Then as $\lam \to 0$, 
\begin{align} 
& \Gg_\lam(x-y)-N^{-1}g^{-1}
\la \widehat{\Gg}_{\lam, Y} (x) ,(1+ F)^{-1} \widehat{\Gg}_{\lam, Y} (y) \ra 
\notag \\
& = 
G_0(x-y)- N^{-1}
\big(\la \hat{G}_{0,Y}(x), \hat{\bf 1} \ra 
+ \la \hat{\bf 1}, \hat{G}_{0,Y}(y) \ra 
+{N^{-1}} \la \hat{\bf 1}, \tD  \hat{\bf 1} \ra \big)  +O(g^{-1}) \lbeq(83)
\end{align}
{ which } is bounded in $\Bb_{\s}$ as $\lam \to 0$. 
\end{corollary} 
\begin{proof}
Substitute $\Gg_\lam(x-y)= g(\lam) + G_0(x-y)+ O(\lam^2 g(\lam)|x-y|^2)$.
Then \refeq(83) immediately follows \reflm(first-term). 
\end{proof}

The second terms in the first formulas for 
$\Gamma(\lam)$ in Lemmas \ref{lm:5-1}, \ref{lm:5-2}, \ref{lm:5-4}, and \ref{lm:5-5}  
are all {sandwiched} by $(1+ F)^{-1}S $ and 
$S(1+F)^{-1}$ and, for studying their  
contributions to $D(\lam,x,y)$, we use the following lemma. 
Recall that $T$ is a linear map defined in $S\C^N$ and if we {identify}  
$T$ with ${ {\bf 0}_{P\C^N}}\oplus T$, then $TS= ST= T$. 

\begin{lemma} \lblm(1FG) Let $\s>1$ and $0<\d<\s-1$. Then, 
there exists $\lam_0>0$ such that for $0<\lam<\lam_0$ { the} 
following estimates are satisfied for { some}  constant $C>0$:   
\begin{align} \lbeq(SGL)
& \|S(\widehat{\Gg}_{\lam, Y} - 
\widehat{G}_{0,Y})\|_{L^2_{-\s}(\R^2)} \leq C \lam^\d , \\
& \|S(1+ F)^{-1}\widehat{\Gg}_{\lam, Y} - 
S(\widehat{G}_{0,Y} + N^{-1} \tD \hat{\bf 1}) 
\|_{L^2_{-\s}} \leq C \la g\ra^{-1} \lbeq(1F)
\end{align} 
In particular, { the} $L^2_{-\s}(\R^2)$-valued analytic functions 
$S\widehat{\Gg}_{\lam, Y}$ and 
$S (1+ F)^{-1}\widehat{\Gg}_{\lam, Y}$  of $\lam \in \C^{+}$ have continuous extensions 
to the closure $\Cb^{+}$. 
\end{lemma}
\begin{proof} 
Since $S \hat{g}=0$, we have 
$S(\widehat{\Gg}_{\lam, Y} - \widehat{G}_{0,Y})= S\widehat{R}_{0,Y}(\lam, x)$ 
and \refeq(SGL) follows from \refeq(R-0Y).  We write 
\bqn \lbeq(65)
S (1+ F)^{-1}\widehat{\Gg}_{\lam,Y} = 
S (\widehat{\Gg}_{\lam,Y} - F \widehat{\Gg}_{\lam,Y} + F^{2}(1+ F)^{-1}\widehat{\Gg}_{\lam,Y}) 
\eqn 
and, on the right hand side, { we} substitute \refeq(64) for first two 
$\widehat{\Gg}_{\lam,Y}$, 
$F = -N^{-1}g^{-1} \tD$ in the second term, use $S\hat{g}=0$ and 
arrange so that { in the formula below} the terms in the first line 
are independent of $\lam$, { while} those in the second line are bounded 
by $C \la g(\lam)^{-1}\ra $ in $L^2_{-\s}$ as $\lam \to 0$:  
\begin{align*}
& \refeq(65) = S (\widehat{G}_{0,Y} +  N^{-1}\tD \hat {\bf 1} 
\\
& \qquad +  N^{-1}g^{-1} \tD \widehat{G}_{0,Y} + \widehat{R}_{0,Y}(\lam) + 
F \widehat{R}_{0,Y}(\lam) + F^{2}(1+ F)^{-1}\widehat{\Gg}_{\lam,Y}). 
\end{align*}
This proves \refeq(1F). 
\end{proof}

We now start proving each statement of \refth(zerolimit) separately. 
By virtue of \refcor(first-term), we only have to study  
$\la \widehat{\Gg}_{\lam,Y}(x), \Ga(\lam)^{-1} \widehat{\Gg}_{\lam,Y}(y) \ra $ 
when $\Ga(\lam)^{-1}$ is replaced by the second terms in the first formulas for 
$\Gamma(\lam)$ in Lemmas \ref{lm:5-1}, \ref{lm:5-2}, \ref{lm:5-4}, and \ref{lm:5-5} for the corresponding cases. 

\subsection{{ Proof of Theorem \ref{th:zerolimit}(1)}} 

\paragraph{} We use \reflm(5-1). When 
$\Ga(\lam)^{-1}$ is replaced by \refeq(G1), 
$\la \widehat{\Gg}_{\lam,Y}(x), \Ga(\lam)^{-1} \widehat{\Gg}_{\lam,Y}(y) \ra $ 
becomes 
\[
\la S(1+ F)^{-1}S \widehat{\Gg}_{\lam,Y}(x), 
\Big([S\tD S]^{-1}+ O(g^{-1})\Big)S
(1+ F)^{-1}S \widehat{\Gg}_{\lam,Y}(y) \ra + O(\lam^2 g)
\]
where we used that 
$\|\widehat{\Gg}_{\lam,Y}\|_{L^2_{-\s}}\leq C \la g \ra$ to 
obtain the term $O(\lam^2 g)$. Statement (1) immediately follows by 
applying \reflm(1FG). 

\subsection{{ Proof of Theorem \ref{th:zerolimit}(2)}} 

\paragraph{Proof of statement (2-a)}

We apply \reflm(5-2) and \refcor(first-term). 
Replacing $\Ga(\lam)^{-1}$ in 
$\la \widehat{\Gg}_{\lam,Y}(x), \Ga(\lam)^{-1} \widehat{\Gg}_{\lam,Y}(y) \ra $ by \refeq(S1) 
produces 
\bqn \lbeq(75)
\la (1+ F)^{-1}S \widehat{\Gg}_{\lam,Y}(x), 
T\Big(g [T\tD^2 T]^{-1}+ O(1)\Big)T
(1+ F)^{-1}S \widehat{\Gg}_{\lam,Y}(y) \ra +O(\lam^2 g^3).
\eqn 
Thus, if $T = \fb \otimes \fb$ with normalized $\fb\in \C^N$ and 
$\la \fb, \tD^2 \fb \ra= \c_0^{-2}$, $\c_0>0$ then \reflm(1FG) implies that 
\[
\refeq(75) = \c_0^2 g \ph(x) \ph(y) + O(1), \quad 
\ph(x)= \la \widehat{G}_{0,Y}(x) + N^{-1} \tD \hat{\bf 1},\fb \ra 
\]
Here $f_1+ \cdots+ f_N=0$ as $\fb=(f_1, \dots, f_N)\in S\C^N$ and 
\begin{multline}
\la \fb, \hat{G}_{0,Y}(x) \ra = - \frac1{2\pi} \sum f_j \log(|x-y_j|) 
=- \frac1{2\pi} \sum f_j \Big(\log(|x-y_j|) -\log |x|\Big) \\
=\sum f_j \Big(\int_0^1 \frac{(x-\th y_j)y_j}{2\pi |x-\theta y_j|^2}d\theta\Big)  
= \sum f_j \frac{y_j \cdot x}{|x|^2} + O(|x|^{-2}).  \lbeq(m-1)
\end{multline}
In the matrix representation of $\tD$ in \reflm(td2t), $\eb$ is represented 
by $\begin{pmatrix} 1 \\ {\bf 0} \end{pmatrix}$ and ${\bf \hat{1}}$ by 
$\begin{pmatrix} \sqrt{N} \\ {\bf 0} \end{pmatrix}$. It follows  that 
$\tD \hat{\bf 1}$ is given by the vector 
$\begin{pmatrix} a \\ {\bf a} \end{pmatrix}$.
Then, $T \tD \hat{\bf 1} = T {\bf a}$ which does not vanish 
as $T \tD^2 T$ is non-singular in $T\C^N$ (see the proof of \reflm(td2t)).   
Hence $\la \tD \hat{\bf 1},\fb\ra \not=0$ and this proves statement (2-a).

\paragraph{Proof of statement (2-b)} We use \reflm(5-4) and \refcor(first-term). 
As { before}, we only have to study 
\bqn \lbeq(92)
N^{-1} g^{-1}\lam^{-2}  \la (1+ F)^{-1}S \widehat{\Gg}_{\lam,Y}(x), 
T[T\Gg_1(Y) T]^{-1}T S(1+ F)^{-1} 
\widehat{\Gg}_{\lam,Y}(y) \ra 
\eqn 
and the remainder is of order $O(\lam^{-2}g^{-2})$. 
We diagonalize the symmetric matrix as
\[
T[T\Gg_1(Y)T]^{-1}T = \sum_{j=1}^n a_j {\fb}_j \otimes {\fb}_j , \quad n = {\rm rank}\, T\, ,
\] 
where $a_j \in \R\setminus\{0\}$ and $\fb_j \in T\C^N$, $j=1, \dots, N$ { can be chosen to be real}. Then, 
\reflm(1FG) implies 
\[ 
\refeq(92)= N^{-1} g^{-1}\lam^{-2} \sum_{j=1}^n a_j \ph_j(x)\ph_j(y), \quad 
\ph_j(x)= \la \fb_j, \hat{G}_{0,Y}(x)+ N^{-1}\tD \hat{\bf 1} \ra. 
\] 
Here $\la \fb_j, N^{-1}\tD \hat{\bf 1} \ra=0$ since $\tD\fb_j=0$,    
\begin{align*} 
\ph_j(x)= \la \fb_j, \hat{G}_{0,Y}(x) \ra & = \sum_{k=1}^N  f_{jk} \log(|x-y_k|) 
= -2 \sum_{k=1}^N f_{jk} \frac{y_k \cdot x}{|x|^2} + O(|x|^{-2}).
\end{align*}
{ We must have that}  $\sum f_{jk} y_k \not=0$ for $j=1, \dots, n$ because for every { real vector $\fb\in T\C^N$ we have:}
\bqn \lbeq(TG1Y)
\la T\Gg_1(Y)T\fb, \fb \ra =  -\frac1{4\pi}\sum_{j,k=1}^N |y_j-y_k|^2 f_j f_k 
= \frac1{2\pi}\Big(\sum_{j=1}^N f_j y_j \Big)^2>0
\eqn 
{ where we use the assumption that $T\Gg_1(Y)T$ is non-singular}. 

\paragraph{Proof of statement (2-c)} We  use \reflm(5-5). As in the proof of  
statement (2-b), we only need to study 
\bqn \lbeq(97)
-N^{-1}\lam^{-2} \la \hat{G}_{0,Y}(x), T_1[T_1\Gg_2(Y)T_1]^{-1}T_1 \hat{G}_{0,Y}(y)\ra
\eqn 
and the remainder is $O(\lam^{-2}g^{-1})$. 
If we { diagonalize } 
\[
T_1 \tilde{\Gg}_2(Y)T_1 = \sum_{j=1}^m a_j {\ab}_j \otimes {\ab}_j , \quad m = {\rm rank}\, T_1, 
\] 
then we have 
$$
\refeq(97)= -N^{-1}\lam^{-2} \sum_{j=1}^n {a_j} \p_j(x) \otimes \p_j(y), \quad 
\p_j(x)= \la \ab_j, \hat{G}_{0,Y}(x)\ra.
$$
Here we have $\ab_j ={}^t (a_{j1}, \dots, a_{jN}) \in T_1\C^N \subset T\C^N \subset S\C^n$, hence, 
\[
a_{j1}+ \dots +a_{jN}=0, \quad \tD \ab_j=0, \quad a_{j1}y_1+ \dots +a_{jN}y_N=0
\]
where the last equation is the result of \refeq(TG1Y) and $\la T\Gg_1(Y)T\fb, \fb \ra=0$ 
for $\fb\in T_1\C^N$. It follows from { Proposition}  \ref{lemahoria1} that $\psi_j(x)$, $j=1, \dots, m$ 
are all eigenfunctions of $H_{\a, Y}$ with eigenvalue zero. 
This completes the proof of \refth(zerolimit). \qed 

\section{Proof of Theorem \ref{th:1}}\label{section4}
We only prove the theorem for $W_{+}$. 
The complex conjugation $u\mapsto \mathcal{C}u=\overline{u}$ then 
gives the proof for $W_{-}= \mathcal{C}^\ast W_{+} \mathcal{C}$. 
In what follows we assume $H_{\a,Y}$ is of regular type and 
the results of \reflm(5-1) are satisfied. 

\subsection{Stationary representation of the wave operators}
We use the {\it stationary representation} of $W_{+}$ as in the three 
dimensional case (see \cite{DMSY}). We need some preparation. 
We set  
\bqn 
\Dg_\ast= \{u \in \Sg(\R^2) \ | \ \hat u \in C_0^\infty(\R^2 \setminus\{0\})\}.
\eqn 

\begin{lemma} \lblm(density) For { every $n\in\N$}, $\Dg_\ast$ is a dense subspace 
of $L^p(\R^n)$ for all $1 < p<\infty$. 
\end{lemma} 
\begin{proof} 
It suffices to show that { having fixed $f \in \Sg(\R^n)$, then for every $\ep>0$} there exists a $u \in \Dg_\ast$ such that $\|f-u\|_p <\ep$. 
Take a $\chi\in C_0^\infty(\R^n)$ { with $0\leq \chi\leq 1$ }such that $\chi(\xi)=1$ for $|\xi|<1$ and 
$\chi(\xi)=0$ for $|\xi|\geq 2$ and set $\chi_\r(\xi)=\chi(\xi/\r)$.  { If we define }  
$ u = (1-\chi_\r(D))\chi_N(D) f \in \Dg_\ast$ then  
\[
\|f-u\|_p \leq \|f-\chi_N(D)f\|_p + { \|\chi_\r(D)f\|_p }
\]
and  
it suffices to show that $\|(\chi_N(D)-1)f\|_p \to 0$ as $N\to \infty$ 
and $\|\chi_\r (D) f\|_p \to 0$ as $\r \to 0$. To see $\|(\chi_N(D)-1)f\|_p \to 0$ as $N\to \infty$, we write 
\begin{align*}
(\chi_N(D)-1)f(x)
& ={(2\pi)^{-1}}\int_{\R^n} \Big(N^n { (\mathcal{F}^{-1}{\chi})}(N(x-y))f(y)-f(x)\Big) dy \\
& ={(2\pi)^{-1}}\int_{\R^n} { (\mathcal{F}^{-1}{\chi})}(y)(f(x+N^{-1}y)-f(x)) dy \\
&= {(2\pi)^{-1}}\int_{\R^n} N^{-1}y { (\mathcal{F}^{-1}{\chi})}(y)\cdot 
\left(\int_0^1 \nabla f(x+\theta N^{-1}y)d\theta\right) dy 
\end{align*}
and apply Minkowski's inequality to obtain 
\bqn \lbeq(1p)
\|(\chi_N(D)-1)f(x)\|_p \\
\leq  {(2\pi)^{-1}}N^{-1}\||y|{ (\mathcal{F}^{-1}{\chi})}\|_1  \|\nabla f\|_p. 
\eqn 
For the second { limit} we apply Young's inequality and obtain  
\[
\|\chi_\r (D) f \|_p = {(2\pi)^{-1}}\|{ (\mathcal{F}^{-1}{\chi_\r})} \ast f \|_p 
\leq  { (2\pi)^{-1}\r^{n(1-1/p)} \|(\mathcal{F}^{-1}{\chi})\|_p \|f\|_1}. 
\]
\end{proof}

We define the operator $\Omega_{jk}$, $j,k=1,\dots,N$ such that 
$(\Omega_{jk}u)(x)$ for $u \in \Dg_\ast$ is given by 
\[
\frac{1}{\pi\ii}\:\lim_{\delta\downarrow 0}
\int_0^{+\infty} 
\lambda\,e^{-\delta\lambda}
\overline{(\Gamma_{\alpha,Y}(\lambda)^{-1})_{jk}}\, \mathcal{G}_{-\lambda}(x)
\left(\int_{\R^2} 
\big(\mathcal{G}_{\lambda}(y)- 
\mathcal{G}_{-\lambda}(y)\big) u(y) \ud y\right) \ud\lambda\, .
\]
Then the following lemma may be proved by repeating line by line the proof 
of Proposition 3.2 of \cite{DMSY} for the corresponding formula in three dimensions.  
\begin{lemma} Let $(T_{x_0}f)(x):=f(x-x_0)$ be the translation operator 
by $x_0$. Then, for $u, v \in \Dg_\ast$,  
\begin{equation}\label{eq:stationary_representation}
( W^{+}_{\a, Y} u,v)\;=\;( u,v ) 
+\sum_{j,k=1}^N ( T_{y_j}
\Omega_{jk}T_{y_k}^{\,*}u,v)\, . 
\end{equation}
\end{lemma}

{ In order to prove our theorem it} suffices to show that 
\bqn \lbeq(aim)
\|\W_{jk}u\|_p \leq C \|u\|_p, \quad u \in \Dg_\ast, \quad j,k=1, \dots, N 
\eqn 
for any $1<p<\infty$ and for a constant $C$ independent of $u$. 
We first remark here that the damping factor $e^{-\d\lam}$ in the definition of $\W_{jk} u$ 
is unnecessary. To see this we first note that $(\xi^2-z^2)^{-1}$ 
has a limit in $\Sg'(\R^2)$ as $z \to -\lam +i0$, $\lam>0$ and,  for  $v\in \Dg_\ast$:  
\begin{align}
& \int_{\R^2} \overline{v(x)}\Gg_{-\lam}(x)dx = \lim_{\ep \downarrow 0}
\big\la v, \Gg_{-\lam+i\ep} \big\ra 
= \lim_{\ep \downarrow 0} \frac1{2\pi}\big ( v, {\mathcal F}^\ast  (\xi^2-(-\lam+i\ep)^2)^{-1} \big ) 
\notag 
\\
& \qquad = \lim_{\ep \downarrow 0} \frac1{2\pi}\big ( {\mathcal F}v, (\xi^2-(\lam-i\ep)^2)^{-1} \big ) 
= \lim_{\ep \downarrow 0}\frac1{2\pi}  \int_{\R^2} \frac{\overline{\hat{v}(\xi)}}
{\xi^2-\lam^2 +i\ep}d\xi.  \lbeq(rem-def)
\end{align}
Then, as a function of 
$\lam$  
\begin{multline} \lbeq(poisson)
\int_{\R^2}\big(\mathcal{G}_{\lambda}(y)- 
\mathcal{G}_{-\lambda}(y)\big)u(y)\ud y \\
= \lim_{\ep\downarrow 0}\frac1{2\pi} \int_{\R^2}\Big(
\frac{1}{\eta^2-\lam^2-i\ep}- 
\frac{1}{\eta^2-\lam^2+ i\ep}\Big) \hat{u}(\eta)\ud \eta 
= \frac{i}{2}\int_{{\mathbb S}^1} \hat {u}(\lam \w)d\w  
\end{multline} 
is of class $C_0^\infty((0,\infty))$. Then { we have:}  
\bqn \lbeq(Wjk)
(\Omega_{jk}u)(x)= 
\frac{1}{\pi\ii}
\int_0^{+\infty} \lam \overline{(\Gamma_{\alpha,Y}(\lambda)^{-1})_{jk}}\, \mathcal{G}_{-\lambda}(x)
\left(\int_{{\mathbb S}^1} \hat{u}(\lam \w) d\w \right) \ud\lambda\, . 
\eqn 

\subsection{Decomposition of the operator $\W_{jk}$} 
For { simplicity}  we define for $j,k=1,\dots, N$ 
$$
\tGa_{jk}(\lam)= [\overline{\Gamma_{\alpha,Y}(|\lam|)}^{-1}]_{jk}. 
$$
We let $\tGa_{jk}(|D|)$ and $K$ be the operators defined for $u\in \Dg_\ast$ respectively by  
\begin{gather} 
\tilde{\Ga}_{jk}(|D|) u(x) = \frac1{2\pi}\int_{\R^2} e^{ix\xi}\,
\tGa_{jk}(|\xi|)({\mathcal F}{u})(\xi) d\xi, \lbeq(tga) 
\\
Ku(x) = \frac{1}{\pi\ii}
\int_0^{+\infty} \mathcal{G}_{-\lambda}(x) \lam 
\left(\int_{{\mathbb S}^1} ({\mathcal F}{u})(\lam \w) d\w \right) \ud\lambda\, .\nonumber
\end{gather} 
\begin{lemma} 
\lblm(1) 
For every $j,k=1, \dots, N$ the operator $\Omega_{jk}$ is the product of $\tilde{\Ga}_{jk}(|D|)$ 
and $K$: 
\bqn \lbeq(prod)
(\Omega_{jk}u)(x)= \big(K \circ \tilde{\Ga}_{jk}(|D|)\big) u(x), \quad u\in \Dg_\ast.
\eqn  
\end{lemma}
\begin{proof} We may write the right hand side of \refeq(Wjk) in the form 
$$
\frac{1}{\pi\ii}
\int_0^{+\infty} {\lam} \mathcal{G}_{-\lambda}(x)
\left(\int_{{\mathbb S}^1} 
[\overline{\Gamma_{\alpha,Y}(\lambda)}^{-1}]_{jk} ({\mathcal F}{u})(\lam \w) d\w \right) \ud\lambda .
$$
Here $[\overline{\Gamma_{\alpha,Y}(\lambda)}^{-1}]_{jk}({\mathcal F}{u})(\lam \w)= 
{\mathcal F}(\tilde{\Ga}_{jk}(D) u)(\lam\w)$ by the definition of 
$\tilde{\Ga}_{jk}(|D|)$. The lemma follows. 
\end{proof} 

\subsection{Estimate of $Ku$} 
In what follows we shall prove that both $K$ and $\tGa_{jk}(|D|)$, $j,k =1, \dots, N$ are 
bounded operators from 
$L^p(\R^2)$ to itself for $1<p<\infty$. We deal with $K$ first. 

\begin{lemma} 
\lblm(2) 
For any $1<p<\infty$, there exists a constant $C>0$ such that 
\[
|\la v, Ku \ra |\leq C \|u\|_p \|v\|_{p'}, \quad u ,v\in \Dg_\ast 
\]
and $K$ extends to a bounded operator from $L^p$ to itself. 
\end{lemma} 
\begin{proof} 
Let $u, v \in \Dg_\ast$. Define a signed measure $\mu_u$ on $(0,\infty)$ by 
\[
\mu_u (E)= \int_{\lam \in E} \left( \int_{{\mathbb S}^1} \hat{u}(\lam \w) d\w \right) \lam d\lam   
\]
for Borel sets $E$ of $(0,\infty)$. The measure $\m_u$ is supported on a compact 
subset of $(0,\infty)$ and 
\[
(v, Ku) = \frac{1}{\pi{i}}\int_{\R^2} \overline{v(x)} \left( 
\int_0^\infty  \mathcal{G}_{-\lambda}(x)\m_u(d\lam)\right) dx. 
\]
Changing the order of integration by using Fubini theorem, we have 
\bqn \lbeq(2-4-1)
(v, Ku) = \frac{1}{\pi{i}}\int_0^\infty \left( \int_{\R^2}  
\overline{v(x)} \mathcal{G}_{-\lambda}(x) dx \right)\m_u(d\lam). 
\eqn 
Since the limit as $\ep \to 0$ converges uniformly on compact sets of $\lam$ in 
\refeq(rem-def), we may change of order of the limit and the integral in \refeq(2-4-1) 
and, applying Fubini's theorem again we have  
\begin{align}
(v, Ku) & = \lim_{\ep \downarrow 0} \int_0^\infty 
\left(\frac1{2\pi^2{i}}  \int_{\R^2} \frac{\overline{\hat{v}(\xi)}d\xi}
{\xi^2-\lam^2 +i\ep}\right) \m_u(d\lam) \notag \\
& = \lim_{\ep \downarrow 0} 
\frac1{2\pi^2{i}} \int_{\R^2} \overline{\hat{v}(\xi)}    \left( \int_0^\infty 
\frac{\m_u(d\lam)}
{\xi^2-\lam^2 +i\ep} \right) d\xi.   \lbeq(uKv)
\end{align}
Here the inner integral in \refeq(uKv) is equal to  
\[
\int_0^\infty 
\left( \int_{{\mathbb S}^1} \frac{\hat{u}(\lam \w)}{\xi^2-\lam^2 +i\ep} d\w \right) \lam d\lam 
= \int_{\R^2}\frac{\hat{u}(\eta)}{\xi^2-\eta^2 +i\ep} d\eta 
\]
and, Fubini's theorem and the change of variables $(\xi, \eta)$ to $(\xi+ \eta, \eta)$ 
imply 
\[
(v, Ku) = \lim_{\ep \downarrow 0} 
\frac1{2{\pi^2} i} \int_{\R^4}
\frac{\overline{\hat{v}(\xi)}\hat u(\eta)}{\xi^2-\eta^2 +i\ep}d\xi d\eta 
= \lim_{\ep \downarrow 0} \frac{1}{2{\pi^2} i }
\int_{\R^4}\frac{\overline{\hat{v}(\xi+\eta)}\hat{u}(\eta)}{\xi^2+2\xi\eta +i\ep} d\eta d\xi.
\] 
Substituting 
\[
\frac1{\xi^2+2\xi\eta +i\ep}= - i \int_0 ^{\infty} e^{it({\xi^2+2\xi\eta +i\ep})} dt
\]
and using Fubini's theorem once more yield 
\begin{align}\label{hc1}
(v, Ku) = \lim_{\ep \downarrow 0} \frac{-1}{2\pi^2}\int_0^\infty e^{-\ep t} \left\{
\int_{\R^2} e^{it\xi^2} \left(\int_{\R^2}e^{i2t\xi\eta}
\overline{\hat{v}(\xi+\eta)}\hat{u}(\eta)  d\eta  \right) d\xi \right\} dt.  
\end{align}
Apply Parseval identity to the inner most integral and  
change variables $(x,\xi) \to (x, (y-x)/2t)$. Then the function inside 
$\{\cdots \}$ becomes  
$$
\int_{\R^2}e^{it\xi^2} \left(\int_{\R^2}e^{ix\xi}\overline{v(x)} u(x+2t\xi)dx\right) 
d\xi
= \int_{\R^4}e^{i(y^2-x^2)/4t} \overline{v(x)}{u}(y) t^{-2} dx dy. 
$$
{ Introduce this identity in \eqref{hc1}}  and change $t \to 1/4t$: 
\bqn \lbeq(2-3) 
(v, Ku)= \lim_{\ep \downarrow 0} \frac{-2}{\pi^2}\int_0^\infty e^{-\ep/4t} 
\left(\int_{\R^2}e^{-itx^2}\overline{v(x)} dx \right)
\left(\int_{\R^2}e^{ity^2}{u}(y)dy\right) dt. 
\eqn 
Now we introduce the spherical mean: 
\bqn \lbeq(smean)
M_u(r)= \frac{1}{2\pi} \int_{{\mathbb S}^1}u(r\w)d\w, \quad r>0,  
\eqn 
define $N_u(r)=M_u(\sqrt{r})$ for $r>0$ and $N_u(r)=0$ for $r\leq 0$ 
and let ${\mathcal R}$ be the restriction operator to the positive half line: 
\[
({\mathcal R} { f} )(r)= \left\{ 
\begin{array}{cl}  
{ f}(r),&  \quad  r>0,\\
0 ,  & \quad  r \leq 0. 
\end{array}
\right.
\]
Using polar coordinates, we then have 
\begin{align*}
\int_{\R^2}e^{ity^2}{u}(y)dy&= 
2\pi \int_0^\infty e^{itr^2} M_u(r)rdr \\ 
&= \pi \int_{\R} e^{itr} N_u(r)dr
= \sqrt{2\pi}\pi ({\mathcal F}^\ast N_u)(t), 
\end{align*}
where ${\mathcal F}$ is the one dimensional Fourier transform and, likewise 
\[
\int_{\R^2}e^{itx^2}{v}(x)dx= \sqrt{2\pi}\pi ({\mathcal F}^\ast N_v)(t).
\]
Since $({\mathcal F}^\ast N_u)(t), ({\mathcal F}^\ast N_v)(t) \in L^2(\R)$, the limit as $\ep \to 0$ 
in \refeq(2-3) can be trivially taken and Parseval's identity implies 
\begin{align*}
(v, Ku)& = -4\pi \int_0^\infty 
(\overline{{\mathcal F}^\ast N_v)(t)} ({\mathcal F}^\ast N_u)(t) dt \\
& = -4\pi \big ( {\mathcal F}^\ast N_v , {\mathcal R}{\mathcal F}^\ast N_u \big )
= - 4\pi \big ( N_v, {\mathcal F}{\mathcal R}{\mathcal F}^\ast N_u \big ). 
\end{align*} 
As is well known, the operator  
\[
u(x) \mapsto ({\mathcal F}{\mathcal R}{\mathcal F}^\ast u)(x) 
= \frac{i}{2\pi} \int_{\R}\frac{u(y)}{x-y -i0}dy 
\]
is bounded in $L^p(\R)$ for any $1<p<\infty$. Thus H\"older's inequality implies 
\begin{align*}
(v, Ku)& \absleq C 
\left(\int_0^\infty |M_v(\sqrt{r})|^{p'} dr \right)^{1/p'} 
\left(\int_0^\infty |M_u(\sqrt{r})|^p dr \right)^{1/p} \\
& \leq  C \|v\|_{L^{p'}(\R^2)}\|u\|_{L^p(\R^2)}. 
\end{align*}
This completes the proof. 
\end{proof} 

\noindent 
{\bf Remark.} Equation \refeq(2-3) and 
the argument following it imply that 
$$
Ku(x) = \lim_{\ep \downarrow 0} \frac{2i}{\pi^2}\int_{\R^2} \frac{u(y)dy}{x^2- y^2-i\ep} .
$$

\subsection{Proof of \refthb(1), the case $N=1$}

Thanks to \reflms(1,2), it suffices to prove that $\tGa_{jk}(|D|)$, 
$j,k=1, \dots, N$ are bounded from $L^p(\R^2)$ to itself for $1<p<\infty$. 
We recall { Mikhlin's} multiplier theorem (\cite{Stein}): 

\begin{lemma} \lblm(3) 
Let $k>n/2$ be an integer. Suppose $m \in C^k(\R^n \setminus\{0\})$ and 
\bqn \lbeq(cond-Mihlin)
\pa_\xi^\a m(\xi) \absleq C_\a |\xi|^{-|\a|}, \quad |\a|\leq k. 
\eqn
Then, the Fourier multiplier $m(D)$ by $m(\xi)$ defined by 
\[
m(D)u(x) = \frac1{(2\pi)^{n/2}} \int_{\R^n} e^{ix\xi} m(\xi) {\mathcal F}u(\xi) d\xi 
\]
is bounded from $L^p(\R^n)$ to  itself for all $1<p<\infty$. 
\end{lemma} 

When $N=1$ we have for $\lam>0$ that 
\[
\tGa(\lam)= \overline{\Gamma (\lam)}^{-1}
= \Big(\alpha+ \frac1{2\pi}\log \Big(\frac{\lam}{2}\Big)  - \frac{i}{4} 
+ \frac{\gamma_0}{2\pi}\Big)^{-1} .
\]
It is obvious that $\tGa(\lam)\in C^\infty((0,\infty))$ and 
\bqn \lbeq(G-Mikhlin)
\tGa^{(\ell)} (\lam) \absleq C_\ell \lam^{-\ell}, \quad \ell=0,1, \dots, 
\eqn 
which implies 
\[
\pa_{\xi}^\a \tGa (|\xi|) \absleq C_\ell |\xi|^{-|\a|}, \quad |\a|\geq 0 .
\]
Hence, \reflm(3) implies $\tGa(|D|)$ is bounded from $L^p(\R^2)$ 
to itself for any $1<p<\infty$. This completes the proof of \refth(1) 
for the case $N=1$. 

\subsection{{Proof of \refthb(1), the case $N \geq 2$}} 
For $N \geq 2$, $\Ga(\lam)$ has the form  
$$\Ga(\lam)= \{\a\}- g(\lam)- { J}(\lam), \quad 
{ J}(\lam)= \Big(\Gg_\lam(y_j-y_k)\hat{\d}_{jk}\Big) . $$
where $\{\a\}$ is the diagonal matrix with entries $\a_1, \dots, \a_N$ 
and $g(\lam)$ is the scalar matrix $g(\lam){\bf 1}$. Thus, for $N \geq 2$, 
$\Ga(\lam)$ contains the term $\Gg_\lam (y_j-y_k)$ which is oscillatory 
for large $\lam$. 
This prevents to directly apply \reflm(3) 
to $\tGa_{jk}(|D|)$ and we need to split it into the low and high energy 
parts and treat them differently. Recall \refeq(large-hankel) that 
\bqn \lbeq(g-ke)
\mathcal{G}_\lam(x)= e^{i\lam|x|}\w(\lam|x|)
\eqn 
where $\w(\lam)$ satisfies for $\lam>c >0$, $c$ being any positive number, 
\bqn \lbeq(sym-1/2)
\pa_\lam ^\ell \w(\lam)\absleq C_\ell \la \lam \ra^{-\frac12-\ell},  \quad 
\ell=0,1,2 \dots. 
\eqn

\paragraph{\bf Low energy estimate of $\tGa(|D|)$} 
If $H_{\a,Y}$ is of regular type, then \reflm(5-1) shows that 
$\left((\Ga(\lam)^{-1}\right)_{jk}$ satisfies \refeq(8-1) 
hence, so does $\tGa_{jk}(\lam)$, $j,k=1,\dots, N$. 
It follows that if 
$\chi \in C_0^\infty(\R^1)$ is such that $\chi(\lam)=1$ for 
$|\lam|\leq \lam_0/2$ and $\chi(\lam)=0$ for $|\lam|\geq \lam_0$, 
$\lam_0$ being as  in \refeq(8-1), 
then $\chi(|\xi|) \tGa_{jk}(|\xi|)$ satisfies the condition of \reflm(3) 
and $\chi(|D|)\tGa_{jk}(|D|)$ is 
bounded from $L^p(\R^2)$ to itself for all $1<p<\infty$. 

\paragraph{\bf High energy estimate of $\tGa(|D|)$.} 
Thus, the proof of \refth(1) will be completed if we prove 
$(1-\chi(|D|))\tGa_{jk}(|D|)$ is also bounded from 
$L^p(\R^2)$ to itself for all $1<p<\infty$. We use the following result 
due to Peral (\cite{peral}, page 139). 

\begin{lemma}[Peral] \lblm(6) 
The translation invariant Fourier integral operator 
\bqn \lbeq(Peral)
(Tf)(x) = \frac1{(2\pi)^{n/2}}
\int_{\R^n} e^{ix\xi + i |\xi|}\frac{\psi(\xi)}{|\xi|^{b}}
\hat{f}(\xi) d\xi,  
\eqn 
where $\psi(\xi)\in C^\infty(\R^n)$ is such that $\psi(\xi)=0$ 
in a neighbourhood of $\xi=0$ and $\psi(\xi)=1$ for $|\xi|\geq 2$, 
is bounded in $ L^p(\R^n)$ if and only if 
\bqn 
\left\vert \frac1{p}- \frac1{2} \right\vert \leq \frac{b}{n-1}\, . 
\eqn 
\end{lemma}

\begin{lemma}\lblm(7) Let $\chi\in C_0^\infty(\R)$ be as above. 
Then, we may write 
\bqn \lbeq(dec-)
(1-\chi(\lam))\tGa(\lam)= \big({ \Phi}_{jk}(\lam)+ L_{jk}(\lam)\big)_{jk}, 
\eqn 
where ${ \Phi}=({ \Phi}_{jk})$ and $L=(L_{jk})$ satisfy the following properties:  
\begin{enumerate}
\item[{\rm (1)}]  For $j,k=1, \dots, N$, ${ \Phi}_{jk}(\lam)$ is of the form  
\bqn \lbeq(Fjk){ \Phi}_{jk}(\lam)= \sum_{\ell=1}^M e^{ia_\ell \lam } b_\ell (\lam),  
\eqn 
where $a_1, \dots, a_M>0$ are constants 
and $b_1, \dots, b_M$ are symbols of order $-1/2$ on $\R$ 
(which, of course, depend on $j,k$ but we suppress such dependence as the argument 
will be the same for all $j,k$).  
\item[{\rm (2)}] For $j,k=1, \dots, N$, $L_{jk}(\lam)$  satisfy 
\bqn \lbeq(Ljk)
\pa^\ell_{\lam} L_{jk}(\lam) \absleq C_\ell \la \lam \ra^{-2}, \quad \ell=0,1, \cdots. 
\eqn   
\end{enumerate}
\end{lemma} 
\begin{proof} Since $(1-\chi(\lam))\tGa(\lam)_{jk}$ is smooth, it suffices to prove that 
the decomposition \refeq(dec-) is possible for { $\lam \gg 1$}. As 
$(\{\a\}- g(\lam))^{-1} \to 0$ as $\lam \to \infty$ and 
\bqn \lbeq(Djk)
\pa_\lam^\ell { J}_{jk}(\lam) \leq C_\ell \la \lam \ra^{-\frac12}, \quad \ell=0,1, \dots 
\eqn 
for large $\lam$ by virtue of \refeq(g-ke), we may write 
\[
(1-\chi(\lam))\tGa(\lam)= (1-\chi(\lam))\big(\{\a\}- \overline{g(\lam)}\big)^{-1}
\big(1 -(\{\a\}- \overline{g(\lam)})^{-1}\overline{{ J}(\lam)}\big)^{-1}, 
\]
which implies 
\bqn \lbeq(8-5) 
\pa_\lam ^{\ell} (1-\chi(\lam))\tGa(\lam)_{jk} \leq C_\ell,  \quad \ell=0,1, \dots. 
\eqn
We expand 
$\big(1 -(\{\a\}- \overline{g})^{-1}\overline{{ J}}\big)^{-1}$ and define ${ \Phi}(\lam)$ and $L(\lam)$ as 
\begin{gather*}
{ \Phi}(\lam)= (1-\chi(\lam))\sum_{k=0}^3 (\{\a\}- \overline{g(\lam)})^{-1}
\Big(\overline{{ J}(\lam)}(\{\a\}- \overline{g(\lam)})^{-1}\Big)^k , \\  
L(\lam)= (1-\chi(\lam)){ \Ga(\lam)^{-1} } (\{\a\}- \overline{g}(\lam))^{-4} \overline{{ J}(\lam)}^{4} . 
\end{gather*}
Then we have $(1-\chi(\lam))\tGa(\lam)^{-1}= { \Phi}(\lam)+ L(\lam)$ and  
the entries of ${ \Phi}(\lam)$ satisfy the property \refeq(Fjk) for large $\lam>0$. 
The entries of 
$L(\lam)$ satisfy \refeq(Ljk) by virtue of \refeq(8-5) and \refeq(Djk) which implies 
${ J}(\lam)^{4}$ and its derivatives are bounded by $\la \lam \ra^{-2}$ . 
The lemma follows. 
\end{proof} 

We have 
\[
(1-\chi(|D|))\tGa_{jk}(|D|)= { \Phi}_{jk}(|D|) + {L}_{jk}(|D|).
\]
Then by virtue of \refeq(Fjk) \reflm(6) with $b=1/2$ and $n=2$ implies 
${ \Phi}_{jk}(|D|)$ is bounded from $L^p(\R^2)$ to itself; 
Mikhlin's \reflm(3) shows that the operator ${L}_{jk}(|D|)$ is bounded 
from $L^p(\R^2)$ to itself by virtue of \refeq(Ljk). This concludes the 
proof of \refth(1) also for $N \geq 2$.


\begin{thebibliography}{10}

\bibitem{Ag} {\sc S. ~Agmon}, 
\textit{Spectral properties of Schr\"odinger
operators and scattering theory}, Ann. Scuola 
Norm. Sup. Pisa Cl. Sci. (4)
 {\bf 2} (1975), 151--218. 


\bibitem{albeverio-solvable}
{\sc S.~Albeverio, F.~Gesztesy, R.~H{\o}egh-Krohn, and H.~Holden}, 
{ {\em  Solvable Models in Quantum Mechanics}. Second Edition. AMS Chelsea Publishing, Providence, RI, (2005). }

\bibitem{DLMF} {\em Digital Library of Mathematical Functions} https://dlmf.nist.gov/ 

\bibitem{DMSY} 
{\sc G. Dell'Antonio, A.~Michelangeli, R. Scandone and K. Yajima}, 
{\em {The $L^p$-boundedness of wave operators for the three-dimensional 
multi-centre point interaction}}. 
Ann. Inst. H. Poincar\'e {\bf 19} (2018) 283-322. 

\bibitem{EG} 
{\sc M. ~B. ~Erdo{\u g}an and W. ~R. ~Green}, 
{\em {Dispersive estimates for Schr\"odinger operators in dimension two with 
obstructions at zero energy}}. Trans. Amer. Math. Soc. {\bf 365} (2013), 6403-6440. 

{
\bibitem{EGG} 
{\sc M. ~B. ~Erdo{\u g}an,  Michael Goldberg and W. ~R. ~Green}, 
{\em {On the $L^p$ boundedness of wave operators for two-dimensional Schr\"odinger operators with threshold obstructions}}.  J. Funct. Anal. {\bf 274} (2018),  2139-2161.

}

\bibitem{DMW}
{\sc V.~Duch{\^e}ne, J.~L. Marzuola, and M.~I. Weinstein}, {\em {Wave operator
  bounds for one-dimensional {S}chr{\"o}dinger operators with singular
  potentials and applications}}, J. Math. Phys., {\bf 52} (2011), pp.~013505, 17.



\bibitem{JN}
{\sc A. Jensen and G. Nenciu} 
{\em {A unified approach to resolvent expansions at thresholds}}, 
Reviews in Mathematical Physics, {\bf 13}, No. 6 (2001) 717-754. 

\bibitem{Kato} 
{\sc T. ~Kato}, {\em {Perturbation of Linear Operators}}, Springer Verlag. 
Heidelberg-New-York-Tokyo (1966).

\bibitem{Ku} S. ~T. ~Kuroda, {\it Introduction 
to Scattering Theory}, 
Lecture Notes, Matematisk Institut, Aarhus University (1978). 

\bibitem{Stein} 
{\sc E. ~M.~Stein}, {\em {Harmonic Analysis: Real-Variable Methods, 
Orthogonality, and oscillatory Integrals} }, Princeton U. Press, 
Princeton, N. J. (1993). 

\bibitem{peral} 
{\sc J. ~C. ~Peral}, {\em {$L^p$ estimate for the wave equation} }, 
J. Funct. Anal. {\bf 36}, 114--145 (1980). 


\bibitem{Watson} 
{\sc G.~N. ~Watson} {\em {Theory of Bessel functions}}, Cambridge Univ. Press,
London (1922) .



\bibitem{Ya} 
{\sc K. ~Yajima} {\em {Remarks on {$L^p$}-boundedness of wave operators for
  {S}chr{\"o}dinger operators with threshold singularities}}, Doc. Math.,{\bf 21}, 
   391--443 (2016).


\bibitem{Yosida}
{\sc K.~Yosida} {\em {Functional Analysis, 6-th printing}}, Springer-Verlag, 
Heidelberg-New-York- Tokyo (1980) 

\end{thebibliography}
\end{document}